\documentclass{article}

\usepackage[utf8]{inputenc} 
\usepackage[T1]{fontenc}    
\usepackage{url}            
\usepackage{booktabs}       
\usepackage{amsfonts}       
\usepackage{nicefrac}       
\usepackage{microtype}      

\usepackage{custom_tex}
\usepackage{xcolor}
\definecolor{ed}{RGB}{225,0,0}

\definecolor{sifan}{rgb}{0.0, 0.0, 1.0}

\usepackage{diagbox}
\usepackage{grffile}
\usepackage{xr}

\newcommand{\numberthis}{\addtocounter{equation}{1}\tag{\theequation}}
\usepackage{multicol}
\usepackage{multirow}
\usepackage{makecell, rotating, diagbox}
\usepackage{slashbox}
\usepackage{array}
\usepackage[export]{adjustbox}
\usepackage{subcaption}

\title{Asymptotics for Sketching in Least Squares Regression}

\date{\today}
\author{Edgar Dobriban\footnote{Wharton Statistics Department, University of Pennsylvania. E-mail: \texttt{dobriban@wharton.upenn.edu}.} \, and Sifan Liu\footnote{Department of Statistics,
Stanford University. E-mail: \texttt{sfliu@stanford.edu}. The bulk of this work was performed while SL was a student at Tsinghua University.} }

\begin{document}

\maketitle

\begin{abstract}
  We consider a least squares regression problem where the data has been generated from a linear model, and we are interested to learn the unknown regression parameters. We consider "sketch-and-solve"  methods that randomly project the data first, and do regression after. Previous works have analyzed the statistical and computational performance of such methods. However, the existing analysis is not fine-grained enough to show the fundamental differences between various methods, such as the Subsampled Randomized Hadamard Transform (SRHT) and Gaussian projections. In this paper, we make progress on this problem, working in an asymptotic framework where the number of datapoints and dimension of features goes to infinity. We find the limits of the accuracy loss (for estimation and test error) incurred by popular sketching methods. We show separation between different methods, so that SRHT is better than Gaussian projections. Our theoretical results are verified on both real and synthetic data. The analysis of SRHT relies on novel methods from random matrix theory that may be of independent interest.
\end{abstract}

\section{Introduction}
To enable learning from large datasets, randomized algorithms such as sketching or random projections are an effective approach of wide applicability \citep{mahoney2011randomized, woodruff2014sketching,drineas2016randnla}. In this work, we study the statistical performance of sketching algorithms in linear regression. Various versions of this fundamental problem have been studied before \citep[see e.g.,][and the references therein]{drineas2006sampling,drineas2011faster,dhillon2013new,ma2015statistical,raskutti2014statistical,thanei2017random}. Specifically, in a generative model where the data are sampled from a linear regression model, \citet{raskutti2014statistical} have recently compared the statistical performance of various sketching algorithms, such as Gaussian projections and subsampled randomized Hadamard transforms (SRHT) (introduced earlier in \citet{sarlos2006improved,ailon2006approximate}).

However, the known results are not precise enough to enable us to distinguish between the various sketching methods. For instance, the statistical performance of Gaussian projections and SRHT is predicted to be the same \citep{raskutti2014statistical}, whereas the SRHT has been observed to work better in practice \citep{mahoney2011randomized, woodruff2014sketching,drineas2016randnla}. To address this issue, in this paper we introduce a new approach to studying sketching in least squares linear regression. As a key difference from prior work we adopt a "large-data" asymptotic limit, where the relevant dimensions and sample sizes tend to infinity, and can have arbitrary aspect ratios. By leveraging very recent results from asymptotic random matrix theory and free probability theory, we get more accurate results for the performance of sketching.  

We study many of the most popular and important sketching methods in a unified framework, including random projection methods (Gaussian and iid projections, uniform orthogonal---Haar---projections, subsampled randomized Hadamard transforms) as well as random sampling methods (including uniform, randomized leverage-based, and greedy leverage sampling). 
We find clean formulas for the accuracy loss of these methods, compared to standard least squares. As an improvement over prior work, our formulas are accurate down to the constant. We verify these results in extensive simulations and on two empirical datasets. 


\subsection{Problem setup}
\label{probset}

\begin{table}[]
\renewcommand{\arraystretch}{1}
\centering
\caption{Summary of main results. We have a linear model $Y = X\beta+\ep$ of size $n \times p$ and do regression after sketching on data $(SY,SX)$. We show the increase in three loss functions due to sketching: $VE$ (variance efficiency--increase in parameter estimation error), $PE$ (prediction efficiency), and $OE$ (out-of-sample prediction efficiency). The assumptions for $X$ depend on the sketching method.}
\label{summaryresults}
\begin{tabular}{>{\centering\arraybackslash}m{2cm}>{\centering\arraybackslash}m{2cm}>{\centering\arraybackslash}m{2cm}>{\centering\arraybackslash}m{2cm}>{\centering\arraybackslash}m{4cm}}
\toprule[1pt]
Assumption on $X$ & Arbitrary & Arbitrary & Ortho-invariant & Elliptical: $WZ\Sigma^{1/2}$ \\
\midrule[.8pt] 
Assumption on $S$ & iid entries & Haar/Hadamard & Uniform sampling & Leverage sampling \\
\midrule[.8pt]
$VE$ & \multirow{2}{*}[-0.7em]{\(\displaystyle1+\frac{n-p}{r-p}\)} & \multirow{2}{*}[-0.7em]{\(\displaystyle\frac{n-p}{r-p}\)} & \multirow{2}{*}[-0.7em]{\(\displaystyle\frac{n-p}{r-p}\)} & $\frac{\eta_{sw^2}^{-1}(1-p/n)}{\eta_{w^2}^{-1}(1-p/n)}$ \\
\cmidrule[.8pt]{1-1}\cmidrule[.8pt]{5-5}
$PE$ & & & & $1+\E[w^2(1-s)]\frac{\eta_{sw^2}^{-1}(1-p/n)}{p/n}$ \\
\midrule[.8pt]
$OE$ & \(\displaystyle\frac{nr-p^2}{n(r-p)}\) & \(\displaystyle\frac{r(n-p)}{n(r-p)}\) & \(\displaystyle\frac{r(n-p)}{n(r-p)}\) & $\frac{1+\E{w^2}\eta_{sw^2}^{-1}(1-\gamma)}{1+\E{w^2}\eta_{w^2}^{-1}(1-\gamma)}$ \\
\bottomrule[1pt]
\end{tabular}
\end{table}

Suppose we observe $n$ datapoints $(x_i,y_i)$, $i=1,\ldots,n$, where $x_i$ are the $p$-dimensional features (or predictors, covariates) of the $i$-th datapoint, and $y_i$ are the continuous outcomes (or responses). We assume the usual linear model $y_i=x_i^\top\beta+\ep_i$, where $\beta$ is an unknown $p$-dimensional parameter. Also $\ep_i$ is the zero mean noise, with entries uncorrelated and of equal variance $\sigma^2$ across samples. In matrix form, we have
$Y = X\beta+\ep,$
where $X$ is the $n \times p$ data matrix with $i$-th row $x_i^\top$, and $Y$ is the $n\times1$ outcome vector with $i$-th entry $y_i$.
Then the usual ordinary least squares (OLS) estimator is
$$
\hbeta = (X^\top X)^{-1} X^\top Y,
$$
if $rank(X)=p$. This estimator is a gold standard when $n>p$, extremely popular in practice, and with many optimality properties. 
However, when $n,p$ are large, say on the order of millions or billions, the natural $O(np^2)$ time-complexity algorithms for computing it can be prohibitively expensive.  Sketching reduces the size of the problem by multiplying $(X, Y)$ by the $r\times n$ matrix $S$ to obtain the \emph{sketched data} $(\tilde X, \tilde Y)=(SX,SY)$. The dimensions are now $r\times p$ and $r\times 1$. Then instead of doing regression of $Y$ on $X$, we do regression of $\tilde{Y}$ on $\tilde{X}$. 
The solution is
\begin{align*}
\hbeta_s = (\tilde X^\top \tilde X)^{-1} \tilde X^\top \tilde Y,
\end{align*}
if $rank(SX)=p$. In the remainder, we assume that both $X$ and $SX$ have full column rank, which happens with probability one in the generic case if $r>p$. The computational cost decreases from $np^2$ to $rp^2$, which is significant if $r\ll n$. In parallel, the statistical error increases. There is a tradeoff between the computational cost and statistical error. The natural question is then, how much does the error increase?  

\paragraph{Error Criteria}
To compare the statistical efficiency of the estimators $\hbeta$ and $\hbeta_s$, we evaluate the relative value of their mean squared error. If we use the full OLS estimator, we incur a mean squared error of $\smash{\E\|\hbeta-\beta\|^2}$. If we use the sketched OLS estimator, we incur a mean squared error of $\smash{\E\|\hbeta_s-\beta\|^2}$ instead. To see how much efficiency we lose, it is natural and customary in classical statistics to consider the relative efficiency, which is their ratio \citep[e.g.][]{van1998asymptotic}. We call this the \emph{variance efficiency} ($VE$), because the MSE for estimation can be viewed as the sum of variances of the OLS estimator. Hence, we define 
\begin{align*}
VE(\hbeta_s,\hbeta)&
=\frac{\E{\|\hbeta_s-\beta\|^2}}{\E{\|\hbeta-\beta\|^2}}.
\end{align*}
This quantity is greater than or equal to unity, so $VE\ge 1$, and \emph{smaller is better}. An accurate sketching method would achieve an efficiency close to unity, $VE\approx 1$. Our goal will be to find VE. For completeness, we also consider the \emph{relative prediction efficiency (PE)}, \emph{residual efficiency (RE)}, and \emph{out-of-sample efficiency (OE)}
\begin{gather*}
PE
=\frac{\E{\|X\hbeta_s-X\beta\|^2}}{\E{\|X\hbeta-X\beta\|^2}},\quad
RE=\frac{\E{\|Y-X\hbeta_s\|^2}}{\E{\|Y-X\hbeta\|^2}},
\quad 
OE
=\frac{\E{(x_t^\top \hbeta_s-y_t)^2}}{\E{(x_t^\top \hbeta-y_t)^2}},
\end{gather*}
where $(x_t,y_t)$ is a test data point generated from the same model $y_t=x_t^\top \beta+\ep_t$, and $x_t,\ep_t$ are independent of $X,\ep$, and only $x_t$ is observable. The PE quantifies the loss of accuracy in predicting the regression function $\E[Y|X]=X\beta$, the RE quantifies the increase in residuals, while the OE quantifies the increase in test error.

\subsection{Our contributions}

We consider a "large data" asymptotic limit, where both the dimension $p$ and the sample size $n$ tend to infinity, and their aspect ratio converges to a constant. The size $r$ of the sketched data is also proportional to the sample size. Specifically $n,p$, and $r$ tend to infinity such that the original aspect ratio converges, $p/n\rightarrow\gamma\in(0, 1)$, while the data reduction factor also converges, $r/n\rightarrow\xi\in(\gamma, 1)$. Under these asymptotics, we find the limits of the relative efficiencies under various conditions on $X$ and $S$. 
This asymptotic setting is different from the usual one under which sketching is studied, where $n \gg r$ \citep[e.g.,][]{mahoney2011randomized, woodruff2014sketching,drineas2016randnla}. However our results are accurate even in that regime. It may be possible to get convergence rates for the projections with iid entries using known results on convergence rates of Stieltjes transforms. 

In practice, we do not think that $n$ or $p$ grow. Instead, for any given dataset with given $n$ and $p$, we use our results with $\gamma=p/n$ as an approximation. If $n,p$ are both relatively large (say larger than 20), then our results are already quite accurate.

It turns out that the different methods have different performance, and they are applicable to different data matrices. Our main results are summarized in Table \ref{summaryresults}. For instance, when $X$ is arbitrary and $S$ is a matrix with iid entries, the variance efficiency is $1+(n-p)/(r-p)$, so estimation error increases by that factor due to sketching. The results are stated formally in theorems in the remainder of the paper.

\paragraph{The formulas are accurate and simple} We observe that our results are accurate, both in simulations and in two empirical data analysis examples, see Section \ref{sec:simulation}. In particular, they go beyond earlier work \citep{raskutti2014statistical} because they are accurate not just up to the rate, but also down to the precise constants, even in relatively small samples (see Section \ref{section:compare mahoney} in the supplemental for a comparison). Moreover, they have simple expressions and do not depend on any un-estimable parameters of the data. 

\paragraph{Separation between sketching methods} Our results enable us to compare the different sketching methods to a greater level of detail than previously known. For instance, in estimation error ($VE$), we have $VE_{\textnormal{iid}}=VE_{\textnormal{Haar}}+1=VE_{\textnormal{Hadamard}}+1$. This shows that estimation error for uniform orthogonal (Haar) random projections and the subsampled randomized Hadamard transform (SRHT) \citep{ailon2006approximate} is less than for iid random projections. This shows a separation between orthogonal and iid random projections.




\paragraph{Tradeoff between computation and statistical accuracy} Each sketching method becomes more accurate as the projection dimension increases. However, this comes at an increased computational cost. We give a summary of the algorithmic complexity and statistical accuracy (variance efficiency) of each method in Section \ref{sec: table tradeoof supp}, as well as a numerical comparison in Section \ref{sec: computation time supp} in the supplement.

As an illustrating example, consider the dataset with $n=10^7$ and $p=10^5$ and we want to use SRHT before doing least squares. Our results show that if we project down to $r<n$ samples, then our test error increases by a factor of $r(n-p)/[n(r-p)]$. Suppose now that we are willing to tolerate an increase of 1.1x in our test error. Setting $r(n-p)/[n(r-p)]=1.1$ gives $r=10^6$. So we can reduce the data size 10x, and only incur an increase of 1.1x in test error! This is a striking illustration of the power of sketching.

\paragraph{Technical contributions}
As a specific technical contribution, our results rely on asymptotic random matrix theory \citep[e.g.,][]{bai2009spectral,couillet2011random,yao2015large}. However, we emphasize that the "standard" results such as the Marchenko-Pastur law are \emph{not} enough. For instance, to study the subsampled randomized Hadamard transform (SRHT), we discovered that we can use the results of \citep{anderson2014asymptotically} on \emph{asymptotically liberating sequences}, see also \citep{tulino2010capacity} for prior work. To our knowledge, this is the first time that these results are used in any statistical learning application. Given the importance of the SRHT, and the notoriously difficult nature of analyzing it, we view this as a technical innovation of broader interest. 

Since there are already many different sketching methods proposed before, we do not attempt to introduce new ones here. Our goal is instead to develop a clear theory. This can lead to an increased understanding of the performance of the various methods, helping practitioners choose between them. Our theoretical framework may also help in analyzing and understanding new methods.



\subsection{Related work}
In this section we review some recent related work. Due to space limitations, we can only mention a small subset of them. For overviews of sketching and random projection methods from a numerical linear algebra perspective, see \citep{halko2011finding,mahoney2011randomized,woodruff2014sketching,drineas2017lectures}. For a theoretical computer science perspective, see \citep{vempala2005random}. 


\citep{drineas2006sampling} show that leverage score sampling leads to better results than uniform sampling. \citep{drineas2012fast}, show furthermore that leverage scores can be approximated fast using the Hadamard transform. \citep{drineas2011faster} propose the fast Hadamard transform for sketching in regression. They prove strong relative error bounds on the realized in-sample prediction error for arbitrary input data. Our results concern a different setting that assumes a generative statistical model.


One of the most related works is \citep{raskutti2014statistical}. They study sketching algorithms from both statistical and algorithmic perspectives. However, they focus on a different setting, where $n\gg r$, and prove bounds on $RE$ and $PE$. For instance, they discover that $RE$ can be bounded even when $r$ is not too large, proving bounds such as $RE \le 1+ 44p/r$ for subsampling and subgaussian projections. In contrast, we show more precise results such as $|RE - r/(r-p)|=o(1)$, (without the constant 44). This holds without additional assumption for iid projections, and under the slightly stronger condition of ortho-invariance for subsampling. We show that these conditions are reasonable, because our results are accurate both in simulations and in empirical data analysis examples. 


Other related works include sketching with convex constraints \citep{pilanci2015randomized}, column-wise sketching \citep{maillard2009compressed,kaban2014new,thanei2017random}, tensor sketching \citep{pham2013fast,diao2017sketching,malik2018low}, subspace embedding for nonlinear kernel mapping \citep{avron2014subspace}, partial sketching \citep{dhillon2013new,ahfock2017statistical}, frequent direction in streaming model \citep{liberty2013simple,huang2018near}, count-min sketch \citep{cormode2005improved}, randomized dimension reduction in stochastic geometry \citep{oymak2015universality}. Sketching also has numerous applications to problems in machine learning and data science, such as clustering \citep{cannings2017random}, hypothesis testing \citep{lopes2011more}, bandits \citep{kuzborskij2018efficient} etc.

\section{Theoretical results}
We present our theoretical results in this section. All proofs are in the supplementary material.

\subsection{Gaussian projection}
\label{section:gaussianSFinite}

For Gaussian random projection, the sketching matrix $S$ is generated from the  Gaussian distribution. An advantage of Gaussian projections is that generating and multiplying Gaussian matrices is \emph{embarrassingly parallel}, making it appropriate for certain distributed and cloud-computing architectures. For the performance of Gaussian sketching, we have the following result. The first part gives exact formulas for the variance, prediction, and out-of-sample efficiencies VE, PE, and OE. The second part simplifies the OE approximation for a special class of design matrices $X$. 

\begin{theorem}[Gaussian projection] Suppose $S$ is an $r\times n$ Gaussian random matrix with iid standard normal entries. Let $X$ be an arbitrary $n\times p$ matrix with full column rank $p$, and suppose that $r-p>1$. Then the efficiencies have the following form
\begin{align*}
&VE(\hbeta_s,\hbeta)
=PE(\hbeta_s,\hbeta)
=1+\frac{n-p}{r-p-1},
\\
&OE(\hbeta_s,\hbeta)=\frac{1+\left[1+\frac{n-p}{r-p-1}\right]x_t^\top (X^\top X)^{-1} x_t}{1+x_t^\top (X^\top X)^{-1} x_t}.
\end{align*}
Second, suppose in addition that $X$ is also random, having the form $X=Z\Sigma^{1/2}$, where $Z\in\R^{n\times p}$ has iid entries of zero mean, unit variance and finite fourth moment, and $\Sigma\in\R^{p\times p}$ is a deterministic positive definite matrix. If the test datapoint is drawn independently from the same population as $X$, i.e. $x_t=\Sigma^{1/2}z_t$, then as $n,p,r$ grow to infinity proportionally, with $p/n\rightarrow\gamma\in(0, 1)$ and $r/n\rightarrow\xi\in(\gamma, 1)$, we have the simple formula for OE
\begin{align*}
\limn OE(\hbeta_s,\hbeta)=\frac{\xi-\gamma^2}{\xi-\gamma}
\approx \frac{nr-p^2}{n(r-p)}.
\end{align*}
\label{gsdx}
\end{theorem}
These results are complementary to \citet{raskutti2014statistical}, who showed that $PE\le 44(1 +n/r)$, $RE\le 1 + 44 p/r$ with fixed probability under slightly different assumptions. 
These formulas have all the properties we claimed before: they are simple, accurate, and easy to interpret. The relative efficiencies \emph{decrease} with $r/n$, the ratio of preserved samples after sketching. This is because a larger number of samples leads to a higher accuracy. Also, when $\xi = \lim r/n=1$, $VE$ and $PE$ reach a minimum of 2. Thus, taking a random Gaussian projection will \emph{degrade the performance of OLS even if we do not reduce the sample size}. This is because iid projections distort the geometry of Euclidean space due to their non-orthogonality. We will see how to overcome this using orthogonal random projections.

The proofs have three stages. The first stage, common to all sketching methods, expresses the VE and other desired quantities in terms of traces of appropriate matices. The second stage involves finding the implicit limit of those traces using random matrix theory, in terms of certain fixed-point equations from the Marchenko-Pastur law. The final stage involves finding the explicit limit. In the Gaussian case, the second and third stages simplify into explicit calculations with the Wishart distribution. 
\subsection{iid projections}
\label{section:iid S}

For iid projections, the entries of $S$ are generated independently from the same distribution (not necessarily Gaussian). This will include \emph{sparse projections} with iid $0,\pm1$ entries, which can speed up computation \citep{achlioptas2001database}. We show that in the "large-data" limit the performance of sketching is the same as for Gaussian projections. This is an instance of \emph{universality}.


\begin{theorem}[Universality for iid projection] Suppose that $S$ has iid entries of zero mean and finite fourth moment. Suppose also that $X$ is a deterministic matrix, whose singular values are uniformly bounded away from zero and infinity. Then as $n$ goes to infinity, while $p/n\rightarrow\gamma\in(0,1)$, $r/n\rightarrow\xi\in(\gamma,1)$, the efficiencies have the limits
\begin{align*}
\limn VE(\hbeta_s, \hbeta)
= \limn PE(\hbeta_s, \hbeta)
&=1+\frac{1-\gamma}{\xi-\gamma}.
\end{align*}
Suppose in addition that $X$ is also random, under the same model as in Theorem \ref{gsdx}. Then the formula for OE given there still holds in this more general case.
\label{theorem:iidS,deterministic X}
\end{theorem}
The proof is based on a Lindeberg exchange argument.

\subsection{Orthogonal (Haar) random projection}
\label{subsection:HaarS}
We saw that a random projection with iid entries will degrade the performance of OLS \emph{even if we do not reduce the sample size}. Matrices with iid entries are not ideal for sketching, because they distort the geometry of Euclidean space due to their non-orthogonality. Is it possible to overcome this using orthogonal random projections?
Here $S$ is a Haar random matrix uniformly distributed over the space of all $r\times n$ partial orthogonal matrices. 

We need the following definition. Recall that for an $n\times p$ matrix $M$ with $n\ge p$, such that the eigenvalues of $n^{-1} M^\top M$ are $\lambda_j$, the \emph{empirical spectral distribution (esd.)} of $M$ is the mixture
$
\frac1p\sum_{j=1}^p\delta_{\lambda_j},
$
where $\delta_\lambda$ denotes a point mass distribution at $\lambda$.

\begin{theorem}[Haar projection] Suppose that $S$ is an $r\times n$ Haar-distributed random matrix. Suppose also that $X$ is a deterministic matrix s.t. the esd. of $X^\top X$ converges weakly to some fixed probability distribution with compact support bounded away from the origin. Then as $n$ tends to infinity, while $p/n\rightarrow\gamma\in(0,1)$, $r/n\rightarrow\xi\in(\gamma,1)$, the efficiencies have the limits 
\begin{align*}
\limn VE(\hbeta_s,\hbeta)
=\limn PE(\hbeta_s,\hbeta)
&=\frac{1-\gamma}{\xi-\gamma}.
\end{align*}
Suppose in addition that the training and test data $X$ and $x_t$ are also random, under the same model as in Theorem \ref{gsdx}. Then $\limn OE(\hbeta_s,\hbeta)=\frac{1-\gamma}{1-\gamma/\xi}.$
\label{theorem:haar S}
\end{theorem}
The proof uses the limiting esd of a product of Haar and fixed matrices.
Orthogonal projections are \emph{uniformly better} than iid projections in terms of statistical accuracy. For variance efficiency, $VE_{\textnormal{iid}}=VE_{\textnormal{Haar}}+1$. However, there is still a tradeoff between statistical accuracy and computational cost, since the time complexity of generating a Haar matrix using the Gram-Schmidt procedure is $O(nr^2)$.

\subsection{Subsampled randomized Hadamard transform}
\label{section:hadamard}
A faster way to do orthogonal projection is the subsampled randomized Hadamard transform (SRHT) \citep{ailon2006approximate}, also known as the Fast Johnson-Lindentsrauss transform (FJLT). This is faster as it relies on the Fast Fourier Transform, and is often viewed as a standard reference point for comparing sketching algorithms. 

An $n\times n$ possibly complex-valued matrix $H$ is called a \emph{Hadamard matrix} if $H/\sqrt{n}$ is orthogonal and the absolute values of its entries are unity, $|H_{ij}|=1$ for $i,j=1,\ldots,n$. 
A prominent example, the \emph{Walsh-Hadamard matrix} is defined recursively by
\begin{align*}
H_n=\left(\begin{array}{cc}H_{n/2} & H_{n/2}\\
H_{n/2} & -H_{n/2}\end{array}
\right),
\end{align*}
with $H_1=(1)$. This requires $n$ to be a power of 2. Another construction is the discrete Fourier transform (DFT) matrix with the $(u,v)$-th entry equal to 
$H_{uv}=n^{-1/2}e^{-2\pi i(u-1)(v-1)/n}$.
Multiplying this matrix from the right by $X$ is equivalent to applying the discrete Fourier transform to each column of $X$, up to scaling. The time complexity for the matrix-matrix multiplication for both the transforms is $O(np\log n)$ due to the Fast Fourier Transform, faster than other random projections. 

Now we consider the subsampled randomized Hadamard transform. Define the $n\times n$ subsampled randomized Hadamard matrix as $S=BHDP$, 
where $B\in\R^{n\times n}$ is a diagonal \emph{sampling matrix} of iid Bernoulli random variables with success probability $r/n$, 
 $H\in\R^{n\times n}$ is a Hadamard matrix, $D\in\R^{n\times n}$ is a diagonal matrix of iid random variables equal to $\pm1$ with probability one half, and $P\in\R^{n\times n}$ is a uniformly distributed permutation matrix. In the definition of $S$, the Hadamard matrix $H$ is deterministic, while the other matrices $B,D$ and $P$ are random. At the last step, we discard the zero rows of $S$, so it becomes an $\tilde r \times n$ orthogonal matrix where $\tilde r \approx r$. We expect the SRHT to be similar to uniform orthogonal projections. The following theorem verifies our intuition.  The proof uses free probability theory  \citep{tulino2010capacity,anderson2014asymptotically}. 
\begin{theorem}[Subsampled randomized Hadamard projection] Let $S$ be an $n\times n$ subsampled randomized Hadamard matrix. Suppose also that $X$ is an $n \times p$ deterministic matrix whose e.s.d. converges weakly to some fixed probability distribution with compact support bounded away from the origin.
Then as $n$ tends to infinity, while $p/n\rightarrow\gamma\in(0,1)$, $r/n\rightarrow\xi\in(\gamma,1)$, the efficiencies have the same limits as for Haar projection in Theorem \ref{theorem:haar S}.
\label{theorem:Hadamard S}
\end{theorem}

\subsection{Uniform random sampling}
\label{section:uniform_sampling}

Fast orthogonal transforms such as the Hadamard transforms are considered as a baseline for sketching methods, because they are efficient and work well quite generally. However, if the data are very uniform, for instance if the data matrix can be assumed to be nearly rotationally invariant, then \emph{sampling methods} can work just as well, as will be shown below.

The simplest sampling method is uniform subsampling, where we take $r$ of the $n$ rows of $X$ with equal probability, with or without replacement. Here we analyze a nearly equivalent method, where we sample each row of $X$ independently with probability $r/n$, so that the expected number of sampled rows is $r$. For large $r$ and $n$, the number of sampled rows concentrates around $r$. 

Moreover, we also assume that $X$ is random, and the distribution of $X$ is \emph{rotationally invariant}, i.e. for any $n\times n$ orthogonal matrix $U$ and any $p\times p$ orthogonal matrix $V$, the distribution of $UXV^\top$ is the same as the distribution of $X$. This holds for instance if $X$ has iid Gaussian entries. Then the following theorem states the surprising fact that uniform sampling performs just like Haar projection. 

\begin{theorem}[Uniform sampling]
Let $S$ be an $n\times n$ diagonal uniform sampling matrix with iid Bernoulli$(r/n)$ entries. Let $X$ be an $n\times p$ rotationally invariant random matrix. Suppose that $n$ tends to infinity, while $p/n\rightarrow\gamma\in(0,1)$, and $r/n\rightarrow\xi\in(\gamma,1)$, and the e.s.d. of $X$ converges almost surely in distribution to a compactly supported probability measure bounded away from the origin. Then the efficiencies have the same limits as for Haar matrices in Theorem \ref{theorem:haar S}.
\label{theorem:uniform sampling}
\end{theorem}

\subsection{Leverage-based sampling}
\label{section:leverage}
Uniform sampling can work poorly when the data are highly non-uniform and some datapoints are more influential than others for the regression fit. In that case, it has been proposed to sample proportionally to the leverage scores $h_{ii}=x_i^\top (X^\top X)^{-1} x_i$. These can be thought of as the "leverage of response value $Y_i$ on the corresponding value $\hat{Y}_i$". 
One can also do greedy leverage sampling, deterministically taking the $r$ rows with largest leverage scores \citep{papailiopoulos2014provable}.

In this section, we give a unified framework to study these sampling methods. Since leverage-based sampling does not introduce enough randomness for the results to be as simple and universal as before, we need to assume some more randomness via a model for $X$. Here we consider the \emph{elliptical model}
\begin{align}
x_i=w_i\Sigma^{1/2}z_i,i=1,\ldots,n,
\label{def:elliptical model}
\end{align}
where the \emph{scale variables} $w_i$ are deterministic scalars bounded away from zero, and $\Sigma^{1/2}$ is a $p\times p$ positive definite matrix. Also, $z_i$ are iid $p\times 1$ random vectors whose entries are all iid random variables of zero mean and unit variance. This model has a long history in multivariate statistics, see \citep{mardia1979multivariate}. If a scale variable $w_i$ is much larger than the rest, then $x_i$ will have a large leverage score. This model allows us to study the effect of unequal leverage scores. Similarly to uniform sampling, we analyze the model where each row is sampled independently with some probability. 

Recall that $\eta$-transform of a distribution $F$ is defined by
$
\eta_F(z)=\int\frac{1}{1+zx}dF(x),
$
for $z\in\mathbb{C}^+$ \citep[e.g.,][]{tulino2004random,couillet2011random}. In the next result, we assume that the scalars $w_i^2$, $i=1,\ldots,n$, have a limiting distribution $F_{w^2}$ as the dimension increases. In that case, the eta-transform is the limit of the leverage scores. First we give a result for arbitrary sampling with probability $\pi_i$ depending only on $w_i$, and next specialize it to leverage sampling.

\begin{theorem}[Sampling for elliptical model]
Suppose $X$ is sampled from the elliptical model defined in \eqref{def:elliptical model}. Suppose the e.s.d. of $\Sigma$ converges in distribution to some probability measure with compact support bounded away from the origin. Let $n$ tend to infinity, while $p/n\rightarrow\gamma\in(0,1)$ and $r/n\rightarrow\xi\in(\gamma,1)$. Suppose also that the $4+\eta$-th moment of $z_i$ is uniformly bounded, for some $\eta>0$.

Consider the sketching method where we sample the $i$-th row of $X$ with probability $\pi_i$ independently, where $\pi_i$ may only depend on $w_i$, and $\pi_i,i=1,\ldots,n$ have a limiting distribution $F_\pi$. Let $s|\pi$ be a Bernoulli random variable with success probability $\pi$, then
\begin{align*}
&\limn VE(\hbeta_s,\hbeta)=\frac{\eta_{sw^2}^{-1}(1-\gamma)}{\eta_{w^2}^{-1}(1-\gamma)},
\,\,\limn OE(\hbeta_s,\hbeta)=\frac{1+\E{w^2}\eta_{sw^2}^{-1}(1-\gamma)}{1+\E{w^2}\eta_{w^2}^{-1}(1-\gamma)}
\\
&\limn PE(\hbeta_s,\hbeta)=1+\frac{1}{\gamma}\E{w^2(1-s)}\eta_{sw^2}^{-1}(1-\gamma),
\end{align*}
where $\eta_{w^2}$ and $\eta_{sw^2}$ are the $\eta$-transforms of $w^2$ (where $w$ is the distribution of scales of $x_i$) and $sw^2$ (where $s$ is defined above), respectively.  Moreover, the expectation is taken with respect to the joint distribution of $s,w^2$ as defined above. In particular for leverage score sampling, $s$ is a Bernoulli variable with success probability $\min[r/p(1-1/(1+w^2\eta_{w^2}^{-1}(1-\gamma)), 1]$.
\label{theorem:sampling elliptical model}
\end{theorem}


If $w_i$-s are all equal to unity, one can check that the results are the same as for orthogonal projection or uniform sampling on rotationally invariant $X$. This is because all leverage scores are nearly equal.  
We specialize this result to greedy sampling in Section \ref{sec: greedy leverage supp} in the supplement.

\section{Simulations and data analysis}
\label{sec:simulation}
We report some simulations to verify our results. In Figure \ref{fig:simulation}, we take $n=2000$, and $p=100$ or 800, respectively. Each row of $X$ is generated iid from $\N(0,I_p)$. The simulation results of VE and the error bar are the mean and one standard deviation over 10 repetitions. We also plot our theoretical results (bold lines) in the figures. The $x$-axis is on a log scale. We observe that the simulation results match the theoretical results very well. Also note that in this case, where the data is uniformly distributed, sampling methods work as well as orthogonal and Hadamard projection, while Gaussian and iid projections perform worse. Additional simulations with correlated t-distributed data and leverage sampling are in Section \ref{sec:nonuniform} and \ref{section:leverage:ex} in the supplement.


\begin{figure}
\centering
\includegraphics[width=0.8\linewidth]{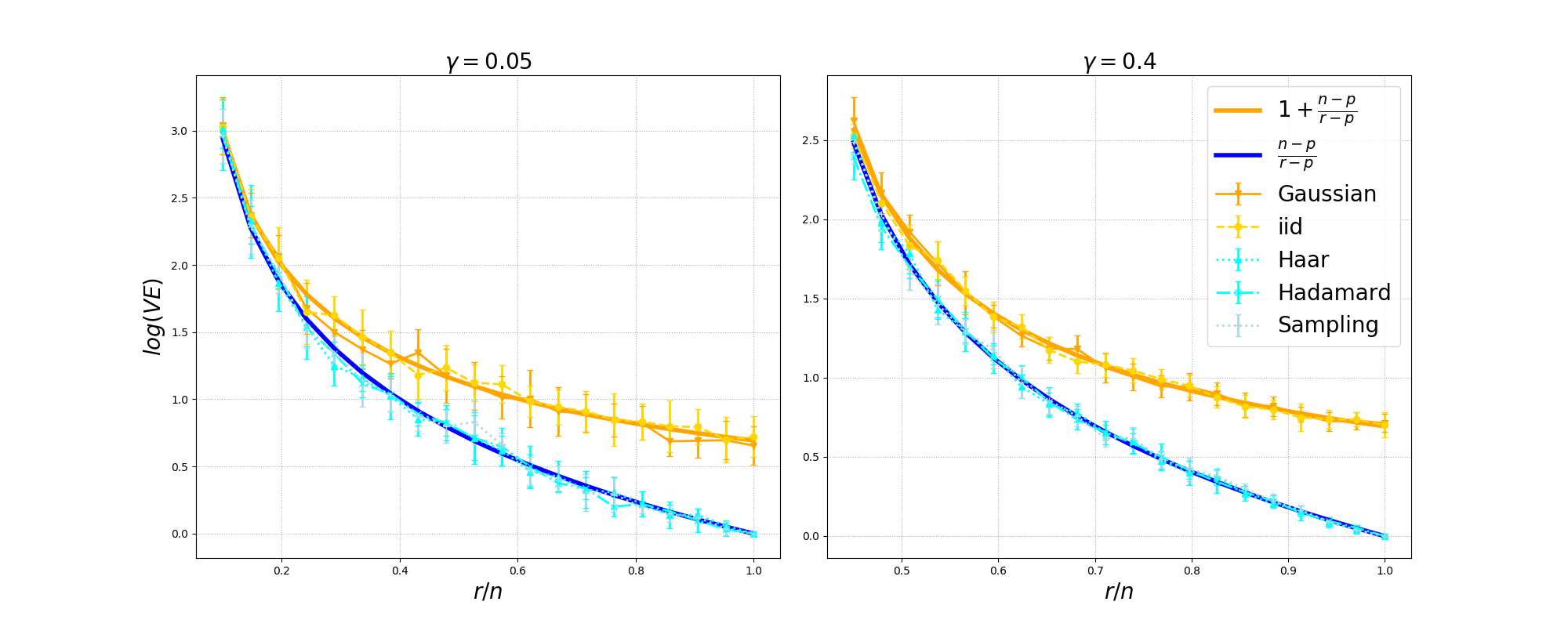}
\caption{Verification of our theory. Solid lines show the theoretical formulas for variance efficiency, while dashed lines show the simulation results, for $\gamma=0.05$ (left, log of VE shown), and $\gamma=0.4$ (right). Showing SD over 10 trials of Gaussian, iid, Haar, Hadamard sketching, and sampling.}
\label{fig:simulation}
\end{figure}
\begin{figure}
\centering
\includegraphics[width=0.8\linewidth]{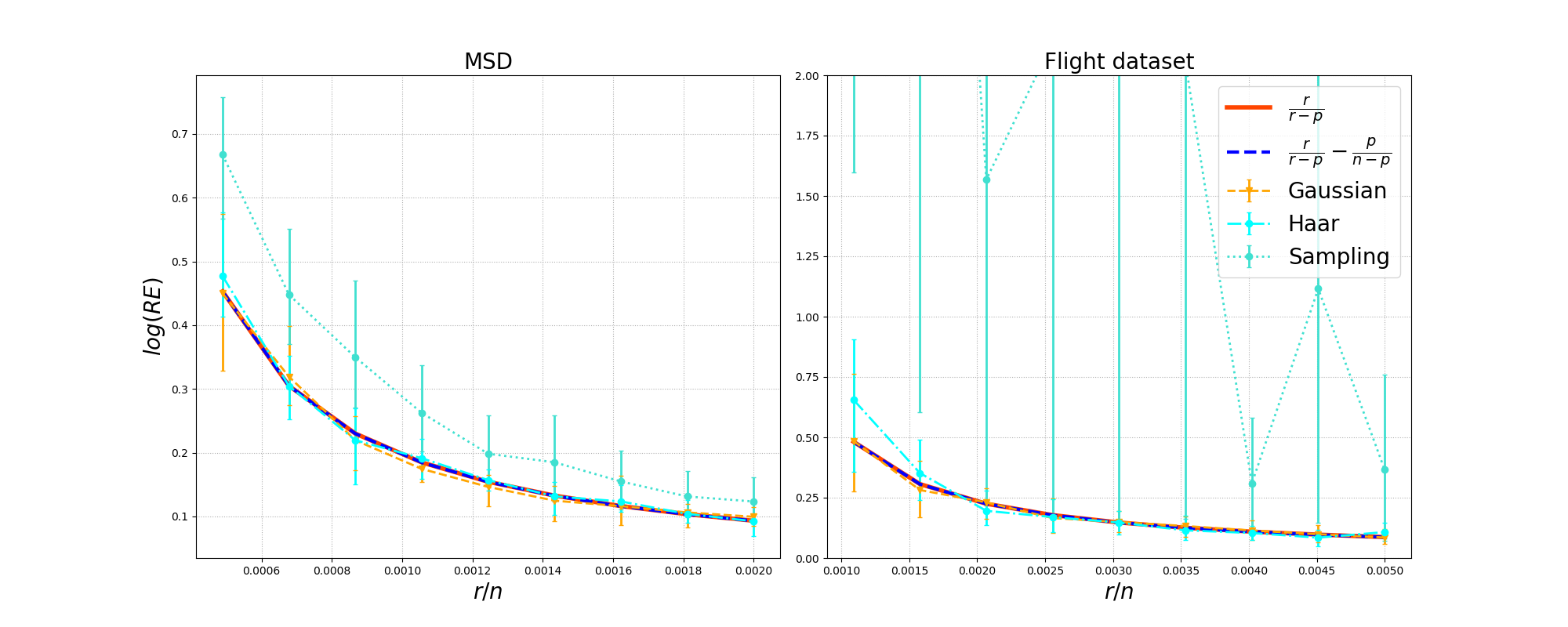}
\caption{Empirical data analysis.  Left: Million Song dataset. Right: Flight dataset.}
\label{fig: empirical}
\end{figure}

We test our results on the Million Song Year Prediction Dataset (MSD) \citep{Bertin-Mahieux2011} ($n=515344$, $p=90$) and the New York flights dataset \citep{nycflights13} ($n=60449$, $p=21$). The columns are standardized to have zero mean and unit standard deviation. 
We compare three different sketching methods: Gaussian projection, randomized Hadamard projection, and uniform sampling. For each target dimension $r$, we show the mean, as well as 5\% and 95\% quantiles over 10 repetitions. The results for RE are in Figure \ref{fig: empirical}, and the results for OE are in Section \ref{sec:oe empirical supp} in the supplement. For Gaussian and Hadamard projections our theory agrees well with the experiments. However, uniform sampling has very large variance, especially on the flight dataset. Our theory is less accurate here, because it requires the data matrix to be rotationally invariant, which may not hold. 



\section*{Discussion}
A direction for future work is to study sketching in (kernel) ridge regression (perhaps possible using RMT), lasso (perhaps possible using approximate message passing). Another question is to understand the variability of sketching methods.




\subsubsection*{Acknowledgments}

The authors thank Ken Clarkson, Miles Lopes, Michael Mahoney, Mert Pilanci, Garvesh Raskutti, David Woodruff for helpful discussions. ED was partially supported by NSF BIGDATA grant IIS 1837992. SL was partially supported by a Tsinghua University Summer Research award. A version of our manuscript is available on arxiv at \url{https://arxiv.org/abs/1810.06089} (the first version had the title "A new theory for sketching in linear regression"), and a short version appears at NeurIPS 2019.

{\small
\setlength{\bibsep}{0.2pt plus 0.3ex}
\bibliographystyle{plainnat-abbrev}
\bibliography{references}

\begin{thebibliography}{54}
\providecommand{\natexlab}[1]{#1}
\providecommand{\url}[1]{\texttt{#1}}
\expandafter\ifx\csname urlstyle\endcsname\relax
  \providecommand{\doi}[1]{doi: #1}\else
  \providecommand{\doi}{doi: \begingroup \urlstyle{rm}\Url}\fi

\bibitem[Achlioptas(2001)]{achlioptas2001database}
D.~Achlioptas.
\newblock Database-friendly random projections.
\newblock In \emph{Proceedings of the twentieth ACM SIGMOD-SIGACT-SIGART
  symposium on Principles of database systems}, pages 274--281. ACM, 2001.

\bibitem[Ahfock et~al.(2017)Ahfock, Astle, and
  Richardson]{ahfock2017statistical}
D.~Ahfock, W.~J. Astle, and S.~Richardson.
\newblock Statistical properties of sketching algorithms.
\newblock \emph{arXiv preprint arXiv:1706.03665}, 2017.

\bibitem[Ailon and Chazelle(2006)]{ailon2006approximate}
N.~Ailon and B.~Chazelle.
\newblock Approximate nearest neighbors and the fast johnson-lindenstrauss
  transform.
\newblock In \emph{Proceedings of the thirty-eighth annual ACM symposium on
  Theory of computing}, pages 557--563. ACM, 2006.

\bibitem[Anderson and Farrell(2014)]{anderson2014asymptotically}
G.~W. Anderson and B.~Farrell.
\newblock Asymptotically liberating sequences of random unitary matrices.
\newblock \emph{Advances in Mathematics}, 255:\penalty0 381--413, 2014.

\bibitem[Anderson et~al.(2010)Anderson, Guionnet, and
  Zeitouni]{anderson2010introduction}
G.~W. Anderson, A.~Guionnet, and O.~Zeitouni.
\newblock \emph{An Introduction to Random Matrices}.
\newblock Number 118. Cambridge University Press, 2010.

\bibitem[Anderson(2003)]{anderson1958introduction}
T.~W. Anderson.
\newblock \emph{An Introduction to Multivariate Statistical Analysis}.
\newblock Wiley New York, 2003.

\bibitem[Avron et~al.(2014)Avron, Nguyen, and Woodruff]{avron2014subspace}
H.~Avron, H.~Nguyen, and D.~Woodruff.
\newblock Subspace embeddings for the polynomial kernel.
\newblock In \emph{Advances in Neural Information Processing Systems}, pages
  2258--2266, 2014.

\bibitem[Bai and Silverstein(2010)]{bai2009spectral}
Z.~Bai and J.~W. Silverstein.
\newblock \emph{Spectral analysis of large dimensional random matrices}.
\newblock Springer Series in Statistics. Springer, New York, 2nd edition, 2010.

\bibitem[Bertin-Mahieux et~al.(2011)Bertin-Mahieux, Ellis, Whitman, and
  Lamere]{Bertin-Mahieux2011}
T.~Bertin-Mahieux, D.~P. Ellis, B.~Whitman, and P.~Lamere.
\newblock The million song dataset.
\newblock In \emph{{Proceedings of the 12th International Conference on Music
  Information Retrieval ({ISMIR} 2011)}}, 2011.

\bibitem[Cannings and Samworth(2017)]{cannings2017random}
T.~I. Cannings and R.~J. Samworth.
\newblock Random-projection ensemble classification.
\newblock \emph{Journal of the Royal Statistical Society: Series B (Statistical
  Methodology)}, 79\penalty0 (4):\penalty0 959--1035, 2017.

\bibitem[Chatterjee(2006)]{chatterjee2006generalization}
S.~Chatterjee.
\newblock A generalization of the lindeberg principle.
\newblock \emph{The Annals of Probability}, 34\penalty0 (6):\penalty0
  2061--2076, 2006.

\bibitem[Cormode and Muthukrishnan(2005)]{cormode2005improved}
G.~Cormode and S.~Muthukrishnan.
\newblock An improved data stream summary: the count-min sketch and its
  applications.
\newblock \emph{Journal of Algorithms}, 55\penalty0 (1):\penalty0 58--75, 2005.

\bibitem[Couillet and Debbah(2011)]{couillet2011random}
R.~Couillet and M.~Debbah.
\newblock \emph{Random {M}atrix {M}ethods for {W}ireless {C}ommunications}.
\newblock Cambridge University Press, 2011.

\bibitem[Couillet and Hachem(2014)]{couillet2014analysis}
R.~Couillet and W.~Hachem.
\newblock Analysis of the limiting spectral measure of large random matrices of
  the separable covariance type.
\newblock \emph{Random Matrices: Theory and Applications}, 3\penalty0
  (04):\penalty0 1450016, 2014.

\bibitem[Dhillon et~al.(2013)Dhillon, Lu, Foster, and Ungar]{dhillon2013new}
P.~Dhillon, Y.~Lu, D.~P. Foster, and L.~Ungar.
\newblock New subsampling algorithms for fast least squares regression.
\newblock In \emph{Advances in neural information processing systems}, pages
  360--368, 2013.

\bibitem[Diao et~al.(2017)Diao, Song, Sun, and Woodruff]{diao2017sketching}
H.~Diao, Z.~Song, W.~Sun, and D.~P. Woodruff.
\newblock Sketching for kronecker product regression and p-splines.
\newblock \emph{arXiv preprint arXiv:1712.09473}, 2017.

\bibitem[Drineas and Mahoney(2016)]{drineas2016randnla}
P.~Drineas and M.~W. Mahoney.
\newblock Rand{NLA}: randomized numerical linear algebra.
\newblock \emph{Communications of the ACM}, 59\penalty0 (6):\penalty0 80--90,
  2016.

\bibitem[Drineas and Mahoney(2017)]{drineas2017lectures}
P.~Drineas and M.~W. Mahoney.
\newblock Lectures on randomized numerical linear algebra.
\newblock \emph{arXiv preprint arXiv:1712.08880}, 2017.

\bibitem[Drineas et~al.(2006)Drineas, Mahoney, and
  Muthukrishnan]{drineas2006sampling}
P.~Drineas, M.~W. Mahoney, and S.~Muthukrishnan.
\newblock Sampling algorithms for l 2 regression and applications.
\newblock In \emph{Proceedings of the seventeenth annual ACM-SIAM symposium on
  Discrete algorithm}, pages 1127--1136. Society for Industrial and Applied
  Mathematics, 2006.

\bibitem[Drineas et~al.(2011)Drineas, Mahoney, Muthukrishnan, and
  Sarl{\'o}s]{drineas2011faster}
P.~Drineas, M.~W. Mahoney, S.~Muthukrishnan, and T.~Sarl{\'o}s.
\newblock Faster least squares approximation.
\newblock \emph{Numerische mathematik}, 117\penalty0 (2):\penalty0 219--249,
  2011.

\bibitem[Drineas et~al.(2012)Drineas, Magdon-Ismail, Mahoney, and
  Woodruff]{drineas2012fast}
P.~Drineas, M.~Magdon-Ismail, M.~W. Mahoney, and D.~P. Woodruff.
\newblock Fast approximation of matrix coherence and statistical leverage.
\newblock \emph{Journal of Machine Learning Research}, 13\penalty0
  (Dec):\penalty0 3475--3506, 2012.

\bibitem[Hachem(2008)]{Hachem2008anexpression}
W.~Hachem.
\newblock An expression for $\int log(t/\sigma^2+1)\mu\boxtimes
  \tilde{\mu}(dt)$.
\newblock \emph{unpublished}, 2008.

\bibitem[Halko et~al.(2011)Halko, Martinsson, and Tropp]{halko2011finding}
N.~Halko, P.-G. Martinsson, and J.~A. Tropp.
\newblock Finding structure with randomness: Probabilistic algorithms for
  constructing approximate matrix decompositions.
\newblock \emph{SIAM review}, 53\penalty0 (2):\penalty0 217--288, 2011.

\bibitem[Hiai and Petz(2006)]{hiai2006semicircle}
F.~Hiai and D.~Petz.
\newblock \emph{The semicircle law, free random variables and entropy}.
\newblock Number~77. American Mathematical Soc., 2006.

\bibitem[Huang(2018)]{huang2018near}
Z.~Huang.
\newblock Near optimal frequent directions for sketching dense and sparse
  matrices.
\newblock In \emph{International Conference on Machine Learning}, pages
  2053--2062, 2018.

\bibitem[Kab{\'a}n(2014)]{kaban2014new}
A.~Kab{\'a}n.
\newblock New bounds on compressive linear least squares regression.
\newblock In \emph{Artificial Intelligence and Statistics}, pages 448--456,
  2014.

\bibitem[Kuzborskij et~al.(2018)Kuzborskij, Cella, and
  Cesa-Bianchi]{kuzborskij2018efficient}
I.~Kuzborskij, L.~Cella, and N.~Cesa-Bianchi.
\newblock Efficient linear bandits through matrix sketching.
\newblock \emph{arXiv preprint arXiv:1809.11033}, 2018.

\bibitem[Liberty(2013)]{liberty2013simple}
E.~Liberty.
\newblock Simple and deterministic matrix sketching.
\newblock In \emph{Proceedings of the 19th ACM SIGKDD international conference
  on Knowledge discovery and data mining}, pages 581--588. ACM, 2013.

\bibitem[Lopes et~al.(2011)Lopes, Jacob, and Wainwright]{lopes2011more}
M.~Lopes, L.~Jacob, and M.~J. Wainwright.
\newblock A more powerful two-sample test in high dimensions using random
  projection.
\newblock In \emph{Advances in Neural Information Processing Systems}, pages
  1206--1214, 2011.

\bibitem[Ma et~al.(2015)Ma, Mahoney, and Yu]{ma2015statistical}
P.~Ma, M.~W. Mahoney, and B.~Yu.
\newblock A statistical perspective on algorithmic leveraging.
\newblock \emph{The Journal of Machine Learning Research}, 16\penalty0
  (1):\penalty0 861--911, 2015.

\bibitem[Mahoney(2011)]{mahoney2011randomized}
M.~W. Mahoney.
\newblock Randomized algorithms for matrices and data.
\newblock \emph{Foundations and Trends{\textregistered} in Machine Learning},
  3\penalty0 (2):\penalty0 123--224, 2011.

\bibitem[Maillard and Munos(2009)]{maillard2009compressed}
O.~Maillard and R.~Munos.
\newblock Compressed least-squares regression.
\newblock In \emph{Advances in Neural Information Processing Systems}, pages
  1213--1221, 2009.

\bibitem[Malik and Becker(2018)]{malik2018low}
O.~A. Malik and S.~Becker.
\newblock Low-rank tucker decomposition of large tensors using tensorsketch.
\newblock In \emph{Advances in Neural Information Processing Systems}, pages
  10096--10106, 2018.

\bibitem[Marchenko and Pastur(1967)]{marchenko1967distribution}
V.~A. Marchenko and L.~A. Pastur.
\newblock Distribution of eigenvalues for some sets of random matrices.
\newblock \emph{Mat. Sb.}, 114\penalty0 (4):\penalty0 507--536, 1967.

\bibitem[Mardia et~al.(1979)Mardia, Kent, and Bibby]{mardia1979multivariate}
K.~Mardia, J.~T. Kent, and J.~M. Bibby.
\newblock \emph{Multivariate analysis}.
\newblock Academic Press, 1979.

\bibitem[Nica and Speicher(2006)]{nica2006lectures}
A.~Nica and R.~Speicher.
\newblock \emph{Lectures on the combinatorics of free probability}, volume~13.
\newblock Cambridge University Press, 2006.

\bibitem[Oymak and Tropp(2017)]{oymak2015universality}
S.~Oymak and J.~A. Tropp.
\newblock Universality laws for randomized dimension reduction, with
  applications.
\newblock \emph{Information and Inference: A Journal of the IMA}, 2017.

\bibitem[Papailiopoulos et~al.(2014)Papailiopoulos, Kyrillidis, and
  Boutsidis]{papailiopoulos2014provable}
D.~Papailiopoulos, A.~Kyrillidis, and C.~Boutsidis.
\newblock Provable deterministic leverage score sampling.
\newblock In \emph{Proceedings of the 20th ACM SIGKDD international conference
  on Knowledge discovery and data mining}, pages 997--1006. ACM, 2014.

\bibitem[Paul and Aue(2014)]{paul2014random}
D.~Paul and A.~Aue.
\newblock Random matrix theory in statistics: A review.
\newblock \emph{Journal of Statistical Planning and Inference}, 150:\penalty0
  1--29, 2014.

\bibitem[Paul and Silverstein(2009)]{paul2009no}
D.~Paul and J.~W. Silverstein.
\newblock No eigenvalues outside the support of the limiting empirical spectral
  distribution of a separable covariance matrix.
\newblock \emph{Journal of Multivariate Analysis}, 100\penalty0 (1):\penalty0
  37--57, 2009.

\bibitem[Pham and Pagh(2013)]{pham2013fast}
N.~Pham and R.~Pagh.
\newblock Fast and scalable polynomial kernels via explicit feature maps.
\newblock In \emph{Proceedings of the 19th ACM SIGKDD international conference
  on Knowledge discovery and data mining}, pages 239--247. ACM, 2013.

\bibitem[Pilanci and Wainwright(2015)]{pilanci2015randomized}
M.~Pilanci and M.~J. Wainwright.
\newblock Randomized sketches of convex programs with sharp guarantees.
\newblock \emph{IEEE Transactions on Information Theory}, 61\penalty0
  (9):\penalty0 5096--5115, 2015.

\bibitem[Raskutti and Mahoney(2016)]{raskutti2014statistical}
G.~Raskutti and M.~W. Mahoney.
\newblock A statistical perspective on randomized sketching for ordinary
  least-squares.
\newblock \emph{The Journal of Machine Learning Research}, 17\penalty0
  (1):\penalty0 7508--7538, 2016.

\bibitem[Sarlos(2006)]{sarlos2006improved}
T.~Sarlos.
\newblock Improved approximation algorithms for large matrices via random
  projections.
\newblock In \emph{Foundations of Computer Science, 2006. FOCS'06. 47th Annual
  IEEE Symposium on}, pages 143--152. IEEE, 2006.

\bibitem[Thanei et~al.(2017)Thanei, Heinze, and Meinshausen]{thanei2017random}
G.-A. Thanei, C.~Heinze, and N.~Meinshausen.
\newblock Random projections for large-scale regression.
\newblock In \emph{Big and complex data analysis}, pages 51--68. Springer,
  2017.

\bibitem[Tulino et~al.(2010)Tulino, Caire, Shamai, and
  Verd{\'u}]{tulino2010capacity}
A.~M. Tulino, G.~Caire, S.~Shamai, and S.~Verd{\'u}.
\newblock Capacity of channels with frequency-selective and time-selective
  fading.
\newblock \emph{IEEE Transactions on Information Theory}, 56\penalty0
  (3):\penalty0 1187--1215, 2010.

\bibitem[Tulino and Verd{\'u}(2004)]{tulino2004random}
A.~M. Tulino and S.~Verd{\'u}.
\newblock Random matrix theory and wireless communications.
\newblock \emph{Communications and Information theory}, 1\penalty0
  (1):\penalty0 1--182, 2004.

\bibitem[Van~der Vaart(1998)]{van1998asymptotic}
A.~W. Van~der Vaart.
\newblock \emph{Asymptotic statistics}.
\newblock Cambridge University Press, 1998.

\bibitem[Vempala(2005)]{vempala2005random}
S.~S. Vempala.
\newblock \emph{The random projection method}, volume~65.
\newblock American Mathematical Soc., 2005.

\bibitem[Voiculescu et~al.(1992)Voiculescu, Dykema, and
  Nica]{voiculescu1992free}
D.~V. Voiculescu, K.~J. Dykema, and A.~Nica.
\newblock \emph{Free random variables}.
\newblock Number~1. American Mathematical Soc., 1992.

\bibitem[Wickham(2018)]{nycflights13}
H.~Wickham.
\newblock \emph{nycflights13: Flights that Departed NYC in 2013}, 2018.
\newblock URL \url{https://CRAN.R-project.org/package=nycflights13}.
\newblock R package version 1.0.0.

\bibitem[Woodruff(2014)]{woodruff2014sketching}
D.~P. Woodruff.
\newblock Sketching as a tool for numerical linear algebra.
\newblock \emph{Foundations and Trends{\textregistered} in Theoretical Computer
  Science}, 10\penalty0 (1--2):\penalty0 1--157, 2014.

\bibitem[Yao et~al.(2015)Yao, Bai, and Zheng]{yao2015large}
J.~Yao, Z.~Bai, and S.~Zheng.
\newblock \emph{Large Sample Covariance Matrices and High-Dimensional Data
  Analysis}.
\newblock Cambridge University Press, New York, 2015.

\bibitem[Zhang(2007)]{lixin2007spectral}
L.~Zhang.
\newblock \emph{Spectral analysis of large dimentional random matrices}.
\newblock PhD thesis, 2007.

\end{thebibliography}
}








\numberwithin{equation}{section}
\appendix
\section{Appendix}
\label{appen}

\subsection{Mathematical background}\label{sec:rmt}
In this section we introduce a few needed definitions from random matrix theory and free probability. See \cite{bai2009spectral, paul2014random,yao2015large} for references on random matrix theory and \cite{voiculescu1992free,hiai2006semicircle,nica2006lectures,anderson2010introduction} for references on free probability. The reader interested in the structure of the proofs may skip to the following sections, and refer back to this section when needed.

The data are the $n\times p$ matrix $X$ and contain $p$ features of $n$ samples. 
Recall that for an $n\times p$ matrix $M$ with $n\ge p$, such that the eigenvalues of $n^{-1} M^\top M$ are $\lambda_j$, the empirical spectral distribution (e.s.d.) of $M$ is the cdf of the eigenvalues. Formally, it is the mixture $\frac1p\sum_{j=1}^p\delta_{\lambda_j}$
where $\delta_\lambda$ denotes a point mass distribution at $\lambda$.

The aspect ratio of $X$ is $\gamma_p = p/n$. We consider limits with $p\to\infty$ and 
$\gamma_p\to\gamma\in(0,\infty)$. If the e.s.d. converges weakly, as $n,p,\to\infty$, to some distribution $F$, this is called the limiting spectral distribution (l.s.d.) of $X$.

The Stieltjes transform of a distribution $F$ is defined for complex valued numbers with positive imaginary part, for which $z \in \mathbb{C}^+=\{z \in \mathbb{C}: \mathrm{Imag}(z)>0 \}$ as
\begin{equation*}
m(z) = \int \frac{\, dF(x)}{x - z}.
\end{equation*}

This can be used to define the $S$-transform of a distribution $F$, which is a key tool for free probability. This is defined as the solution to the equation, which is unique under certain conditions (see \cite{voiculescu1992free}),
\begin{equation*}
m_F(\frac{z+1}{zS(z)})=-zS(z).
\end{equation*}
In addition to Stieltjes transform, there are other useful transforms of a distribution. The $\eta$-transform of $F$ is defined by
\begin{equation} 
\eta_F(z)=\int\frac{1}{1+zx}dF(x)=\frac{1}{z}m_F(-\frac{1}{z}).
\label{def_eta_transform}
\end{equation}

Now let us give a typical and key example of a result from asymptotic random matrix theory. Suppose the rows of $X$ are iid $p$-dimensional observations $x_i$, for $i=1,\ldots,n$. Let $\Sigma$ be the covariance matrix of $x_i$. We consider a model of the form $X = Z\Sigma^{1/2}$, where the entries of $Z$ are iid with zero mean and unit variance, and the e.s.d. of $\Sigma$ converges weakly to a probability distribution $H$. Then the Marchenko-Pastur theorem (see \cite{marchenko1967distribution, bai2009spectral}) states that the e.s.d. of the sample covariance matrix ${n}^{-1} X^\top X$ converges almost surely in distribution to a distribution $F_\gamma$, whose Stieltjes transform is the unique solution of a certain fixed point equation. A lot of information can be extracted from this equation, and we will see examples in the proofs.

Random matrix theory is related to free probability. Here we briefly introduce a few concepts in free probability that will be used in the proofs.  A non-commutative probability space is a pair $(\mathcal{A},\tau)$, where $\mathcal{A}$ is a non-commutative algebra with the unit 1 and $\tau:\mathcal{A}\rightarrow\mathbb{R}$ is a linear functional such that  $\tau(1)=1$. If $\tau(ab)=\tau(ba)$ for all $a,b\in\mathcal{A}$, then $\tau$ is called a trace. If $\tau(a^*a)\geq0$, for all $a\in\mathcal{A}$ and the equality holds iff $a=0$, then the trace $\tau$ is called faithful. There is also an inner product, and thus a norm, induced by $\tau$:
$$\langle a,b\rangle=\tau(a^*b),\,\|a\|^2=\langle a,a\rangle.$$
For $a\in\A$ with $a=a^*$, the spectral radius $\rho(a)$ is defined by $\rho(a)=\lim_{k\rightarrow\infty}|\tau(a^{2k})|^{\frac{1}{2k}}$, whenever this limit exists. The elements in $\mathcal{A}$ are called (non-commutative) random variables, and the law (or distribution) of a random variable $a\in\mathcal{A}$ is a linear functional on the polynomial algebra $[X]$ that maps any $P(x)\in[X]$ to $\tau(P(a))$. The connection between the non-commutative probability space and classical probability theory is the spectral theorem, stating that for all $a\in\A$ with bounded spectral radius, there exists a unique Borel probability measure $\mu_a$ such that  for any polynomial $P(x)\in[X]$,
$$\tau(P(x))=\int P(t)d\mu_a(t).$$
We can also define the Stieltjes transform of $a\in\A$ by
$$m_a(z)=\tau((a-z)^{-1})=-\sum_{k=0}^\infty\frac{\tau(a^k)}{z^{k+1}},$$
which is the same as the Stieltjes transform of the probability measure $\mu_a$ associated with $a$. 

Returning to random matrices, one can easily verify that
$$(\A=(L^{\infty-}\otimes M_n(\mathbb{R})),\tau=\frac{1}{n}\mathbb{E}\tr)$$
is a non-commutative probability space and $\tau=\frac{1}{n}\mathbb{E}\tr$ is a faithful trace, where $L^{\infty-}$ denotes the collection of random variables with all moments finite. For $X\in L^{\infty-}\otimes M_n(\mathbb{R})$, the spectral radius is $\|\|X\|_{op}\|_{L^\infty}$, the essential supremum of the operator norm. The probability measure corresponding to the law of $X$ is the expected empirical spectral distribution
$$\mu_X=\frac{1}{n}\mathbb{E}\sum_{i=1}^n\delta_{\lambda_i},$$
where $\lambda_i$-s are the eigenvalues of $X$.

A collection of random variables $\{a_1,\ldots,a_k\}\subset\mathcal{A}$ are said to be freely independent (or just free) if 
$$\tau[\Pi_{j=1}^mP_j(a_{i_j}-\tau(P_j(a_{i_j})))]=0,$$
for any positive integer $m$, any polynomials $P_1,\ldots,P_m$ and any indices $i_1,\ldots,i_m\in[k]$ with no two adjacent $i_j$ equal \cite{voiculescu1992free,nica2006lectures}. A sequence of random variables $\{a_{1,n},\ldots,a_{k,n}\}_{n\geq1}\subset\mathcal{A}$ is said to be asymptotically free if
$$\tau[\Pi_{j=1}^mP_j(a_{i_j,n}-\tau(P_j(a_{i_j,n})))]\rightarrow0,$$
for any positive integer $m$, any polynomials $P_1,\ldots,P_m$ and any indices $i_1,\ldots,i_m\in[k]$ with no two adjacent $i_j$ equal. If $a,b\in\mathcal{A}$ are free, then the law of their product is called their freely multiplicative convolution, and is denoted $a\boxtimes b$. 

A fundamental result is that the $S$-transform of $a\boxtimes b$ equals the products of $S_a(z)$ and $S_b(z)$ \cite{voiculescu1992free,nica2006lectures}. In addition, random matrices with sufficiently independent entries and "near-uniformly" distributed eigenvectors tend to be asymptotically free in the high-dimensional limit. This is a powerful tool to find the l.s.d. of a product of random matrices.


\subsection{Finite-sample results for fixed matrices}
We start with finite-sample results that are true for any fixed sketching matrix $S$. These results will be fundamental in all remaining work. Later, to simplify these results, we will make probabilistic assumptions. First we find a more explicit form of the relative efficiencies.

\begin{proposition}[Finite $n$ results] Taking expectations only over the noise $\ep$ and $\ep_t$, fixing $X$ and $S$, the efficiencies have the following forms:
\begin{align*}
&VE(\hbeta_s,\hbeta)|X,S=
\frac{\tr[Q_1]}{\tr[(X^\top X)^{-1}]},\, PE(\hbeta_s,\hbeta)|X,S
=\frac{\tr[Q_2]}{p},\\
&OE(\hbeta_s,\hbeta)|X,S
=\frac{1+ x_t^\top Q_1 x_t}{1+ x_t^\top (X^\top X)^{-1}x_t},
\end{align*}
where $Q_0= (X^\top S^\top SX)^{-1}X^\top S^\top S$, while
$Q_1= Q_0 Q_0^\top$, and
$Q_2= X Q_1 X^\top$.
\label{proposition:finite n calculation}
\end{proposition}

\begin{proof}
\label{proposition:finite n calculation:pf}
 The OLS before and after sketching give the estimators $\hbeta$ and $\hbeta_s$
\begin{align*}
\hat{\beta}_{full}&=(X^\top X)^{-1}X^\top Y
=\beta+(X^\top X)^{-1}X^\top \ep,\\
\hat{\beta}_{sub}&=(\tilde{X}^\top \tilde{X})^{-1}\tilde{X}^\top \tilde{Y}
=\beta+(\tilde{X}^\top \tilde{X})^{-1}\tilde{X}^\top \ep
=\beta+ Q_0\ep
.
\end{align*}
We define the "hat" matrices
\begin{align*}
&H= X(X^\top X)^{-1}X^\top ,\\
&\tilde{H}= X(\tilde{X}^\top \tilde{X})^{-1}\tilde{X}^\top S
=X(X^\top S^\top SX)^{-1}X^\top S^\top S
= XQ_0.
\end{align*}
These are both projection matrices, i.e., they satisfy the relations
$H^2=H,\tilde{H}^2=\tilde{H}.$ By our assumptions, we have that $\EE[\ep]{\ep}=0_n$, $\EE[\ep]{\ep\ep^\top}=\sigma^2I_n$, $\tr[H]=\tr[\tilde{H}]=p$. Therefore, we can calculate as follows. 
\benum
\item Variance efficiency:
\begin{align*}
\EE[\ep]{\|\hbeta-\beta\|^2}&
=\EE[\ep]{\|(X^\top X)^{-1}X^\top \ep\|^2}
=\sigma^2\tr[(X^\top X)^{-1}]\\
\EE[\ep]{\|\hbeta_s-\beta\|^2}&
=\EE[\ep]{\|Q_0\ep\|^2}=\sigma^2\tr(Q_0Q_0^\top).
\end{align*}
This proves the formula for $VE$.

\item Prediction efficiency:
\begin{align*}
\EE[\ep]{\|X\beta-X\hbeta\|^2}
&=\EE[\ep]{\|H\ep\|^2}
=\sigma^2\tr[H] = p\sigma^2,\\
\EE[\ep]{\|X\beta-X\hbeta_s\|^2}&
=\EE[\ep]{\|\tilde H\ep\|^2}
=\sigma^2\tr[\tilde{H}^\top \tilde{H}]
=\sigma^2\tr(Q_2).
\end{align*}
This finishes the calculation for $PE$.


\item Out-of-sample efficiency: Here $\ep_t$ is the noise in the test datapoint.
\begin{align*}
\EE[\ep,\ep_t]{(y_t-x_t^\top \hbeta)^2}&
=\EE[\ep,\ep_t]{(\ep_t-x_t^\top (X^\top X)^{-1}X^\top \ep)^2}\\
&=\EE[\ep,\ep_t]{\ep_t^2+\ep^\top X(X^\top X)^{-1}x_tx_t^\top (X^\top X)^{-1}X^\top \ep}\\
&=\sigma^2(1+ x_t^\top (X^\top X)^{-1}x_t),\\
\EE[\ep,\ep_t]{(y_t-x_t^\top \hbeta_s)^2}&
=\EE[\ep,\ep_t]{(\ep_t-x_t^\top Q_0\ep)^2}=\sigma^2(1+ x_t^\top  Q_0  Q_0^\top x_t).
\end{align*}
\eenum 
This finishes the proof. 
\end{proof}

The expressions simplify considerably for orthogonal matrices $S$. Suppose that $S$ is an $r\times n$ matrix such that $SS^\top = I_r$, then we have the following result:

\begin{proposition}[Finite $n$ results for orthogonal $S$] When $S$ is an orthogonal matrix, the above formulas simplify to
\begin{align*}
VE
=\frac{\tr[(X^\top S^\top SX)^{-1}]}{\tr[(X^\top X)^{-1}]},\quad
PE
=\frac{\tr[(X^\top S^\top SX)^{-1}X^\top X]}{p},
\end{align*}
\begin{align*}
OE
=\frac{1+ x_t^\top (X^\top S^\top SX)^{-1} x_t}{1+ x_t^\top (X^\top X)^{-1}x_t}.
\end{align*}
\label{proposition:finite n calculation for orthogonal S}
\end{proposition}
\begin{proof}
Since $S$ satisfies $SS^\top=I_r$, we have $(S^\top S)^2=S^\top S$. Thus, $Q_1 = Q_0Q_0^\top = (X^\top S^\top SX)^{-1}$. With this, the results follow directly from Proposition \ref{proposition:finite n calculation}.
\end{proof}
Actually these formulas hold for any $S$ s.t. $X^\top S^\top SX$ is nonsingular and $S^\top S$ is idempotent.

\subsection{Proof of Theorem \ref{gsdx}}
\label{gsdx:pf}
The proof below utilizes the orthogonal invariance of Gaussian matrices and properties of Wishart matrices. For any $X\in\mathbb{R}^{n\times p}$ with $n\ge p$ and with full column rank, we have the singular value decomposition (SVD) $X=U\Lambda V^\top $, where $U\in\mathbb{R}^{n\times p},V\in\mathbb{R}^{p\times p}$ are both orthogonal matrices, while $\Lambda\in\mathbb{R}^{p\times p}$ is a diagonal matrix, whose diagonal entries are the singular values of $X$. Therefore
\begin{align*}
VE(\hbeta_s,\hbeta)&=
\frac{\EE{\tr((X^\top S^\top SX)^{-2}X^\top (S^\top S)^2X)}}{\tr[(X^\top X)^{-1}]}\\
&=\frac{\EE{\tr(\Lambda^{-2}(U^\top S^\top SU)^{-1}U^\top (S^\top S)^2U(U^\top S^\top SU)^{-1})}}{\tr(\Lambda^{-2})},\\
PE(\hbeta_s,\hbeta)&=\frac{\EE{\tr((X^\top S^\top SX)^{-1}X^\top X(X^\top S^\top SX)^{-1}X^\top (S^\top S)^2X)}}{p}\\
&=\frac{\EE{\tr((U^\top S^\top SU)^{-2}U^\top (S^\top S)^2U)}}{p},\\
OE(\hbeta_s,\hbeta)&=\frac{1+ \EE{x_t^\top (X^\top S^\top SX)^{-1}X^\top (S^\top S)^2X(X^\top S^\top SX)^{-1} x_t}}{1+ \EE{x_t^\top (X^\top X)^{-1}x_t}}\\
&=\frac{1+\EE{x_t^\top V\Lambda^{-1}(U^\top S^\top SU)^{-1}U^\top (S^\top S)^2U(U^\top S^\top SU)^{-1}\Lambda^{-1}V^\top x_t}}{1+\EE{x_t^\top V\Lambda^{-2}V^\top x_t}}.
\end{align*}
We can see that the first two relative efficiencies do not depend on the right singular vectors of $X$.

We denote by $U^\perp\in\mathbb{R}^{n\times(n-p)}$ a complementary orthogonal matrix of $U$, such that $UU^\top +U^{\perp}U^{\perp\top}=I_n$. Let $S_1=SU$, $S_2=SU^\perp$, of sizes $r\times p$, and $r\times (n-p)$, respectively. Then $S_1$ and $S_2$ both have iid $\N(0,1)$ entries and they are independent from each other, because of the orthogonal invariance of a Gaussian random matrix. Also note that 
$$
SS^\top =S(UU^\top +U^{\perp}U^{\perp\top})S^\top =S_1S_1^\top +S_2S_2^\top ,
$$
and 
\begin{align*}
S_1^\top S_1&\sim\mathcal{W}_p(I_p,r),\quad S_2S_2^\top\sim\mathcal{W}_r(I_r,n-p),
\end{align*}
where $\mathcal{W}_p(\Sigma,r)$ is the Wishart distribution with $r$ degrees of freedom and scale matrix $\Sigma$. Then by the properties of Wishart distribution \cite[e.g.,][]{anderson1958introduction}, when $r-p>1$, we have
\begin{align*}
\EE{(S_1^\top S_1)^{-1}}=\frac{I_p}{r-p-1},\, \EE{S_2S_2^\top }=(n-p)I_r.
\end{align*}
Hence the numerator of $VE$ equals
\begin{align*}
&\EE{\tr\left(\Lambda^{-2}(U^\top S^\top SU)^{-1}U^\top (S^\top S)^2U(U^\top S^\top SU)^{-1}\right)}\\
&=\EE{\tr\left(\Lambda^{-2}(S_1^\top S_1)^{-1}S_1(S_1S_1^\top +S_2S_2^\top )S_1^\top (S_1^\top S_1)^{-1}\right)}\\
&=\tr\left(\Lambda^{-2}(I_p+\EE{(S_1^\top S_1)^{-1}S_1^\top S_2S_2^\top S_1(S_1^\top S_1)^{-1}})\right)\\
&=\tr\left(\Lambda^{-2}(I_p+\EE{(S_1^\top S_1)^{-1}S_1^\top (n-p)I_pS_1(S_1^\top S_1)^{-1}})\right)\\
&=\tr\left(\Lambda^{-2}(I_p+(n-p)\EE{(S_1^\top S_1)^{-1}})\right)\\
&=\tr\left(\Lambda^{-2}(1+\frac{n-p}{r-p-1})\right),
\end{align*}
and the denominator $\tr[(X^\top X)^{-1}]=\tr[(V\Lambda^2V^\top )^{-1}]=\tr(\Lambda^{-2}),$
so we have
$VE(\hbeta_s,\hbeta)=1+\frac{n-p}{r-p-1}$. This finishes the calculation for VE. See Section \ref{gsdx:pf2} for the remaining details of this theorem. 

\subsection{Proof of Theorem \ref{theorem:iidS,deterministic X}}
\label{theorem:iidS,deterministic X:pf}

The proof idea is to use a Lindeberg swapping argument to show that the results from Gaussian matrices extend to iid matrices provided that the first two moments match.

Since the error criteria are invariant under the scaling of $S$, we can assume without loss of generality that the entries of $S$ are $n^{-1/2}s_{ij}$, where $s_{ij}$ are iid random variables of zero mean, unit variance, and finite fourth moment. We also let $T=n^{-1/2}t_{ij}$, $t_{ij}$ being iid standard Gaussian random variables, for all $i\in[r]$, $j\in[n]$. 

Let $s$ (respectively, $t$) be the $rn$-dimensional vector whose entries are $s_{ij}$ (respectively, $t_{ij}$) aligned by columns. Then there is a bijection from $s$ to $S$, and from $t$ to $T$. We already know that the desired results for $VE$ and $PE$ hold if $S=T$, and they only depend on $\EE{\tr(Q_1)}$ and $\EE{\tr(Q_2)}$.

For $OE$, under the extra assumptions that $X=Z\Sigma^{1/2}$, we already proved in Theorem \ref{gsdx} that
\begin{align*}
&\EE{x_t^\top(\frac{1}{p}X^\top X)^{-1}x_t}-\tr[(\frac{1}{p}Z^\top Z)^{-1}]\xrightarrow{a.s.}0,\\
&\EE{x_t^\top Q_1x_t}-\tr(Q_1)\xrightarrow{a.s.}0,
\end{align*}
so the results for $OE$ will only depend on $\EE{\tr[Q_1]}$ as well. Thus we only need to show that $\EE{\tr[Q_1(S,X)}$ has the same limit as $\EE{\tr[Q_1(T,X)]}$, and $\EE{\tr[Q_2(S,X)}$ has the same limit as $\EE{\tr[Q_2(T,X)]}$, as $n$ goes to infinity.

Since $SX$ has a nonzero chance of being singular, it is necessary first to show the universality for a regularized trace. See Section \ref{lemma:f,g,limits:pf} for the proof of Lemma \ref{lemma:f,g,limits} below. In the rest of the proof, we let $N=rn$.

\begin{lemma}[Universality for regularized trace functionals]
Let $z_n=\frac{i}{n}\in\mathbb{C}$, where $i$ is the imaginary unit. Define the functions $f_N,g_N:\mathbb{R}^{N}\rightarrow\mathbb{R}$ as
\begin{align}
f_N(s)&=\frac{1}{p}\tr[(X^\top S^\top SX-z_nI_p)^{-2}X^\top(S^\top S)^2X],\label{def_f}\\
g_N(s)&=\frac{1}{p}\tr[(X^\top S^\top SX-z_nI_p)^{-1}X^\top X(X^\top S^\top SX-z_nI_p)^{-1}X^\top(S^\top S)^2X],\label{def_g}
\end{align}
Then $\limn|\EE{f_N(s)}-\EE{f_N(t)}|=0,$ $\limn|\EE{g_N(s)}-\EE{g_N(t)}|=0$.
\label{lemma:f,g,limits}
\end{lemma}
Next we show that the regularized trace functionals have the same limit as the ones we want. See Section \ref{lemma:f_infty,g_infty,limits:pf} for the proof.

\begin{lemma}[Convergence of trace functionals]
Define the functions $f_\infty,g_\infty:\R^{N}\rightarrow\R$
\begin{align}
f_{\infty}(s)&=\frac{1}{p}\tr[(X^\top S^\top SX)^{-2}X^\top(S^\top S)^2X]=\frac{1}{p}\tr[Q_1(S,X)],
\\
g_\infty(s)&=\frac{1}{p}\tr[(X^\top S^\top SX)^{-1}X^\top X(X^\top S^\top SX)^{-1}X^\top(S^\top S)^2X]=\frac{1}{p}\tr[Q_2(S,X)].
\end{align}
Then 
\begin{align*}
&\limn|\EE{f_N(s)}-\EE{f_\infty(s)}|=\limn|\EE{f_N(t)}-\EE{f_\infty(t)}|=0,\\
&\limn|\EE{g_N(s)}-\EE{g_\infty(s)}|=\limn|\EE{g_N(t)}-\EE{g_\infty(t)}|=0.
\end{align*}
\label{lemma:f_infty,g_infty,limits}
\end{lemma}
According to lemma \ref{lemma:f,g,limits} and \ref{lemma:f_infty,g_infty,limits}, we know that 
\begin{align*}
&\limn\frac{1}{p}\EE{\tr[Q_1(S,X)]}=\limn\frac{1}{p}\EE{\tr[Q_1(T,X)]},\\
&\limn\frac{1}{p}\EE{\tr[Q_2(S,X)]}=\limn\frac{1}{p}\EE{\tr[Q_2(T,X)]},
\end{align*}
which concludes the proof of Theorem \ref{theorem:iidS,deterministic X}.

\subsection{Proof of Theorem \ref{gsdx}}
\label{gsdx:pf2} 

For the numerator of $OE$, note that
\begin{align*}
&\EE{x_t^\top V\Lambda^{-1}(U^\top S^\top SU)^{-1}U^\top (S^\top S)^2U(U^\top S^\top SU)^{-1}\Lambda^{-1}V^\top x_t}\\
&=\tr[\EE{(S_1^\top S_1)^{-1}S_1^\top(S_1S_1^\top+S_2S_2^\top)S_1(S_1^\top S_1)^{-1}}\Lambda^{-1}V^\top x_tx_t^\top V\Lambda^{-1}]\\
&=\tr[(I_p+\EE{(S_1^\top S_1)^{-1}S_1^\top(n-p)I_rS_1(S_1^\top S_1)^{-1}})\Lambda^{-1}V^\top x_tx_t^\top V\Lambda^{-1}]\\
&=\tr[(I_p+\frac{n-p}{r-p-1}I_p)\Lambda^{-1}V^\top x_tx_t^\top V\Lambda^{-1}]\\
&=(1+\frac{n-p}{r-p-1})x_t^\top V\Lambda^{-2}V^\top x_t.
\end{align*}
Therefore
\begin{align*}
OE(\hbeta_s,\hbeta)&=\frac{1+(1+\frac{n-p}{r-p-1})x_t^\top (X^\top X)^{-1} x_t}{1+x_t^\top (X^\top X)^{-1} x_t}.
\end{align*}
Additionally, if $x_t=\Sigma^{1/2}z_t$ and $X=Z\Sigma^{1/2}$, we have
$x_t^\top(X^\top X)^{-1}x_t=z_t^\top(Z^\top Z)^{-1}z_t.$ 
Since $z_t$ has iid entries of zero mean and unit variance, we have
\begin{align*}
\EE{z_t^\top(Z^\top Z)^{-1}z_t}=\tr[\EE{(Z^\top Z)^{-1}}\EE{z_tz_t^\top}]=\tr[\EE{(Z^\top Z)^{-1}}]
\end{align*}

Note that the e.s.d. of $\frac{1}{n}Z^\top Z$  converges almost surely to the standard $Mar\breve{c}enko-Pastur$ law \citep{marchenko1967distribution,bai2009spectral} whose Stieltjes transform $m(z)$ satisfies the equation
\begin{align*}
m(z)=\frac{1}{1-\gamma-z-z\gamma m(z)}
\end{align*}
for $z\notin [(1-\sqrt{\gamma})^2,(1+\sqrt{\gamma})^2]$. Letting $z=0$, we have $m(0)=1/(1-\gamma)$, thus
\begin{align*}
\tr[(\frac{1}{n}ZZ^\top)^{-1}]\xrightarrow{a.s.}\frac{1}{1-\gamma},\quad
\tr[(\frac{1}{p}ZZ^\top)^{-1}]\xrightarrow{a.s.}\frac{\gamma}{1-\gamma}.
\end{align*}
Therefore $\EE{x_t^\top(X^\top X)^{-1}x_t}\xrightarrow{a.s.}\frac{\gamma}{1-\gamma}$
and almost surely
\begin{align*}
OE(\hbeta_s, \hbeta)&\rightarrow\frac{1+(1+\frac{1-\gamma}{\xi-\gamma})\frac{\gamma}{1-\gamma}}{1+\frac{\gamma}{1-\gamma}}
=\frac{\xi-\gamma^2}{\xi-\gamma}\text{, as }n\rightarrow\infty.
\end{align*}

Similarly for the numerator of $PE$, we have
\begin{align*}
\EE{\tr((U^\top S^\top SU)^{-2}U^\top (S^\top S)^2U)}&=\EE{\tr((S_1^\top S_1)^{-2}S_1^\top (S_1S_1^\top +S_2S_2^\top )S_1)}\\
&=\EE{\tr(I_p+(S_1^\top S_1)^{-2}S_1^\top S_2S_2^\top S_1)}\\
&=p+\tr(\EE{(S_1^\top S_1)^{-2}S_1^\top (n-p)I_rS_1})\\
&=p+(n-p)\tr(\EE{(S_1^\top S_1)^{-1}})\\
&=p+\frac{(n-p)p}{r-p-1},
\end{align*}
therefore
\begin{align*}
PE(\hbeta_s,\hbeta)&=\frac{p+\frac{(n-p)p}{r-p-1}}{p}=1+\frac{n-p}{r-p-1}.
\end{align*}

This finishes the proof.

\subsection{Proof of Theorem \ref{theorem:iidS,deterministic X}}
\label{theorem:iidS,deterministic X:pf2}

\subsubsection{Proof of Lemma \ref{lemma:f,g,limits}}
\label{lemma:f,g,limits:pf}

The proof of this lemma relies on the Lindeberg Principle, similar to the Generalized Lindeberg Principle, Theorem 1.1 of \cite{chatterjee2006generalization}. The first claim shows universality assuming bounded third derivatives. 

\begin{lemma}[Universality theorem]
Suppose $s$ and $t$ are two independent random vectors in $\R^N$ with independent entries, satisfying $\EE{s_{i}}=\EE{t_i}$ and $\EE{s_i^2}=\EE{t_i^2}$ for all $1\leq i\leq N$, and $\EE{|s_i|^3+|t_i|^3}\leq M<\infty$. Suppose $f_N\in C^3(\R^N,\R)$ and $|\frac{\partial^3f_N}{\partial s_i^3}|$ is bounded above by $L_N$ for all $1\leq i\leq N$ and almost surely as $N$ goes to infinity,
then 
\begin{align*}
&|\EE{f_N(s)-f_N(t)}|= O(L_NN),\text{ as }N\rightarrow\infty.
\end{align*}
\label{lemma:universality chatterjee}
\end{lemma}

The lemma below shows that the third derivatives are actually bounded for our functions of interest, and that the $L_N$ are of order $N^{-3/2}$. 

Since we know the singular values of $X$ are uniformly bounded away from zero and infinity, there exists a constant $c>0$, such that 
\begin{align*}
\frac{1}{c}\leq\sigma_{\min}(X)\leq\sigma_{\max}(X)\leq c.
\end{align*}

\begin{lemma}[Bounding the third derivatives]
Let $f_N(s)$ and $g_N(s)$ be defined in \eqref{def_f} and \eqref{def_g}, where the entries of $s$ are independent, of zero mean, unit variance and finite fourth moment. Then there exists some constant $\phi=\phi(c,\xi,\gamma)>0$, such that  for any partial derivative $\partial_\alpha=\frac{\partial}{\partial_{ij}}$, $\forall i\in[r],j\in[n]$,
\begin{align*}
&|\partial_\alpha^3 f_N|\leq \phi N^{-5/4},\quad
|\partial_\alpha^3 g_N|\leq \phi N^{-5/4}
\end{align*}
hold almost surely as $n$ goes to infinity.
\label{lemma:bound_third_derivatives}
\end{lemma}
The above two lemmas conclude the proof of Lemma \ref{lemma:f,g,limits}. Next we prove them in turn.

\begin{proof}(Proof of Lemma \ref{lemma:universality chatterjee})
The main idea of this proof is borrowed from the proof of Theorem 1.1 of \cite{chatterjee2006generalization}. For each fixed $N$, We write
\begin{align*}
s=(s_1,\ldots,s_N),\quad t=(t_1,\ldots,t_N).
\end{align*}
For each $i=0,1,\ldots,N$, define
\begin{align*}
z_{i}=(s_1,\ldots,s_{i-1},s_i,t_{i+1},\ldots,t_N),\\
z_{i}^0=(s_1,\ldots,s_{i-1},0,t_{i+1},\ldots,t_N).
\end{align*}
Note that $z_0=t,z_N=s$. By a Taylor expansion, we have almost surely that
\begin{align*}
&|f_N(z_i)-f_N(z_i^0)-\partial_if_N(z_i^0)s_i-\frac{1}{2}\partial_i^2f_N(z_i^0)s_i^2|\leq\frac{1}{6}L_N|s_i|^3,\\
&|f_N(z_{i-1})-f_N(z_i^0)-\partial_if_N(z_i^0)t_i-\frac{1}{2}\partial_i^2f_N(z_i^0)t_i^2|\leq\frac{1}{6}L_N|t_i|^3.
\end{align*}
Thus
\begin{align*}
|f_N(z_i)-f_N(z_{i-1})-\partial_if_N(z_i^0)(s_i-t_i)-\frac{1}{2}\partial_i^2f_N(z_i^0)(s_i^2-t_i^2)|\leq\frac{1}{6}(|s_i|^3+|t_i|^3)L_N.
\end{align*}
Since 
\begin{align*}
f_N(s)-f_N(t)=\sum_{i=1}^Nf_N(z_i)-f_N(z_{i-1}),
\end{align*}
we have
\begin{align*}
|f_N(s)-f_N(t)-\sum_{i=1}^N\partial_if_N(z_i^0)(s_i-t_i)-\sum_{i=1}^N\frac{1}{2}\partial_i^2f_N(z_i^0)(s_i^2-t_i^2)|\leq\sum_{i=1}^N\frac{1}{6}(|s_i|^3+|t_i|^3)L_N
\end{align*}
almost surely as $N$ goes to infinity. By the bounded convergence theorem, and because the first two moments of $s,t$ match, we have
\begin{align*}
|\EE{f_N(s)-f_N(t)}|\leq\frac{1}{6}\EE{(|s_i|^3+|t_i|^3)}L_NN,
\end{align*}
thus
\begin{align*}
|\EE{f_N(s)-f_N(t)}|\leq O(L_NN).
\end{align*}
This proves Lemma \ref{lemma:universality chatterjee}.
\end{proof}

\begin{proof}(Proof of Lemma \ref{lemma:bound_third_derivatives})
We will show that the third derivative of $f_N$ and $g_N$ are both bounded in magnitude by $N^{-5/4}$, or equivalently, $n^{-5/2}$. For any $\alpha=(i,j)\in[r]\otimes[n]$, denote $\partial_\alpha=\frac{\partial}{\partial_{ij}}$. Define
$$
G_n(S)=(X^\top S^\top SX-z_nI_p)^{-2}X^\top(S^\top S)^2X,$$
then we have $f_N(s)=\frac{1}{p}\tr(G_n(S))$ and
\begin{align}
(X^\top S^\top SX-z_nI_p)^2G_n(S)=X^\top(S^\top S)^2X.
\label{def_G}
\end{align}
Take derivative w.r.t. $\alpha$ on both sides and we get
\begin{align}
\partial_\alpha[(X^\top S^\top SX-z_nI_p)^2]\cdot G_n(S)+(X^\top S^\top SX-z_nI_p)^2\cdot\partial_\alpha G_n(S)=\partial_\alpha [X^\top(S^\top S)^2X].
\label{derivative_1}
\end{align}
We have
\begin{align*}
\partial_\alpha[(X^\top S^\top SX-z_nI_p)^2]&=\partial_\alpha [(X^\top S^\top SX)^2]-2z_n\partial_\alpha (X^\top S^\top SX)\\
&=\partial_\alpha(X^\top S^\top SX)\cdot(X^\top S^\top SX)
+(X^\top S^\top SX)\cdot\partial_\alpha(X^\top S^\top SX)\\
&-2z_n\partial_\alpha(X^\top S^\top SX),
\end{align*}
and
\begin{align*}
\partial_\alpha(X^\top S^\top SX)&=X^\top[\partial_\alpha (S^\top)\cdot S+S^\top\cdot\partial_\alpha S]X=X^\top(n^{-1/2}E_{ji}S+S^\top n^{-1/2}E_{ij})X,
\end{align*}
where $E_{ij}\in\mathbb{R}^{r\times n}$ whose $(i,j)$-th entry is 1 and the rest are all zeros, and $E_{ji}=E_{ij}^\top$. Therefore
\begin{align*}
\partial_\alpha[(X^\top S^\top SX-z_nI_p)^2]&=[X^\top(E_{ji}S+S^\top E_{ij})XX^\top S^\top SX+X^\top S^\top SXX^\top(E_{ji}S+S^\top E_{ij})X\\
&-2z_nX^\top(E_{ji}S+S^\top E_{ij})X]n^{-1/2}.\numberthis\label{partial_1}
\end{align*}
Similarly, 
\begin{align*}
\partial_\alpha[X^\top(S^\top S)^2X]&=X^\top[\partial_\alpha(S^\top S)\cdot(S^\top S)+(S^\top S)\cdot\partial_\alpha(S^\top S)]X\\
&=\left\{X^\top[(E_{ji}S+S^\top E_{ij})(S^\top S)+(S^\top S)(E_{ji}S+S^\top E_{ij})]X\right\}n^{-1/2}.\numberthis\label{partial_2}
\end{align*}
Denoting $P(S)=X^\top S^\top SX$ and $Q(S)=E_{ji}S+S^\top E_{ij}$, substituting \eqref{partial_1},\eqref{partial_2} into \eqref{derivative_1}, we get
\begin{align*}
\partial_\alpha G(S)&=(P(S)-z_nI_p)^{-2}\{X^\top[Q(S)S^\top S+S^\top SQ(S)]X\\
&-[X^\top Q(S)XP(S)+P(S)X^\top Q(S)X-2z_nX^\top Q(S)X]\}G(S)n^{-1/2}.\numberthis\label{first_order}
\end{align*}
Next we will show that the trace of $\partial_\alpha G(S)$ is bounded by $n^{-1/2}$. By the inequality $\|AB\|\leq\|A\|\|B\|$ and the lemma \ref{lemma:bound trace(AB)} below, we only need to show that the sum of the absolute values of the eigenvalues of $Q(S)$ and the spectral norms of 
\begin{align*}
\quad X^\top X,\quad S^\top S,\quad P(S), \quad (P(S)-z_nI_p)^{-2},\quad G(S)
\end{align*}
are all bounded above by some constants only dependent on $c$ and $\xi$.

\begin{lemma}(Trace of products).
Suppose $A,B$ are two $n\times n$ diagonalizable complex matrices, then 
$$|\tr(AB)|\leq|\lambda|_{\max}(A)\sum_{i=1}^n|\mu_i|,$$ 
where $|\lambda|_{\max}(A)$ is the largest absolute value of eigenvalues of $A$ and $\mu_i$ are the eigenvalues of $B$. 
\label{lemma:bound trace(AB)}
\end{lemma}

Note that
\begin{align*}
Q(S)=E_{ji}S+S^\top E_{ij}&=e_{j}S_{i\cdot}+S_{i\cdot}^\top e_{j}^\top,
\end{align*}
where $e_j$ is an $n\times1$ vector with the $j$th entry equal to 1 and the rest equal to 0, $S_{i\cdot}$ is the $i$th row of $S$. The eigenvalues of $Q(S)$ are $S_{ij}\pm\|S_{i\cdot}\|$, according to Lemma \ref{lemma:uv^top+vu^top} below.
\begin{lemma}(Rank two matrices.)
Let $u,v\in\mathbb{R}^n$ and $u^\top v\neq0$, then the nonzero eigenvalues of $uv^\top+vu^\top$ are $u^\top v\pm\|u\|\|v\|$, both with multiplicity 1.
\label{lemma:uv^top+vu^top}
\end{lemma}
First note that $|S_{ij}|\leq\sigma_{\max}(S)$ and $\|S_{i\cdot}\|\xrightarrow{a.s.}1$ by the law of large number. 
It is also known that as $n\rightarrow\infty$ and $r/n\rightarrow\xi$, we have
\begin{align*}
\lambda_{\min}(S^\top S)\xrightarrow{a.s.}{}(1-\sqrt{\xi})^2,\quad
\lambda_{\max}(S^\top S)\xrightarrow{a.s.}{}(1+\sqrt{\xi})^2,
\end{align*}
see \cite{bai2009spectral}. So the sum of the absolute values of the eigenvalues of $Q(S)$ is bounded above by $2(2+\sqrt{\xi})$, almost surely as $n$ tends to infinity. 

By our assumption, the eigenvalues of $X^\top X$ are bounded in the interval $[\frac{1}{c^2},c^2]$.

Suppose the eigenvalues of $X^\top S^\top SX$ are $\lambda_1\geq\ldots\geq\lambda_p$. So almost surely, 
\begin{align*}
\lambda_p\geq\lambda_{\min}(X^\top X)\lambda_{\min}(S^\top S)\geq\frac{1}{c^2}(1-\sqrt{\xi})^2,\\
\lambda_1\leq\lambda_{\max}(X^\top X)\lambda_{\max}(S^\top S)\leq c^2(1+\sqrt{\xi})^2,
\end{align*}
Since the complex matrix $X^\top S^\top SX-z_nI_p$ is diagonalizable, and its eigenvalues are $\lambda_1-z_n, \ldots, \lambda_p-z_n$. Thus the eigenvalues of $(X^\top S^\top SX-z_nI_p)^{-2}$ are $\frac{1}{(\lambda_1-z_n)^2},\ldots,\frac{1}{(\lambda_p-z_n)^2}$. Because $\lambda_i\in\R$, $z_n=i/n$ and $|\lambda_i-z_n|>|\lambda_i|$, the largest absolute eigenvalue of $(X^\top S^\top SX-z_nI_p)^{-2}$ is bounded above by $\frac{1}{\lambda_p^2}$, that is, $\|(P(S)-z_nI_p)^{-2}\|\leq\frac{1}{\lambda_p^2}\leq \frac{c^4}{(1-\sqrt{\xi})^4}$.

We also have
\begin{align*}
\|G(S)\|&\leq\|(P(S)-z_nI_p)^{-2}\|\|X^\top(S^\top S)^2X\|\\
&\leq\frac{c^4}{(1-\sqrt{\xi})^4}c^2(1+\sqrt{\xi})^4=c^6\frac{(1+\sqrt{\xi})^4}{(1-\sqrt{\xi})^4}.
\end{align*}
Thus $\tr[\partial_\alpha G(S)]$ is bounded by $O(n^{-1/2})$. Since $p/n\rightarrow\gamma$, there exists a constant $\phi_1(c,\gamma,\xi)$, such that 
$$|f_N|=\frac{1}{p}|\tr[\partial_\alpha G(S)]|\leq\phi_1(c,\gamma,\xi)n^{-3/2}.$$

Next we will bound the second derivative of $f_N$ from above by $n^{-2}$.
Take the second derivative w.r.t. to $\alpha$ on both sides of \eqref{def_G}, we have
\begin{align*}
\partial_\alpha^2[(X^\top S^\top SX-z_nI_p)^2]\cdot G(S)+2\partial_\alpha[(X^\top S^\top SX-z_nI_p)^2]\cdot\partial_\alpha G(S)+&(X^\top S^\top SX-z_nI_p)^2\partial_\alpha^2G(S)\\
&=\partial_\alpha^2 [X^\top(S^\top S)^2X],\numberthis\label{derivative_2}
\end{align*}
and thus
\begin{align*}
\partial_\alpha^2G(S)&=(X^\top S^\top SX-z_nI_p)^{-2}[\partial_\alpha^2 [X^\top(S^\top S)^2X]-\partial_\alpha^2[(X^\top S^\top SX-z_nI_p)^2]\cdot G(S)-\\
&2\partial_\alpha[(X^\top S^\top SX-z_nI_p)^2]\cdot\partial_\alpha G(S)].\numberthis\label{second_order}
\end{align*}

Using \eqref{partial_1}, we have
\begin{align*}
\partial_\alpha^2[(X^\top S^\top SX-z_nI_p)^2]&=\partial_\alpha[X^\top(E_{ji}S+S^\top E_{ij})XX^\top S^\top SX+X^\top S^\top SXX^\top(E_{ji}S+S^\top E_{ij})X\\
&-2z_nX^\top(E_{ji}S+S^\top E_{ij})X]n^{-1/2}\\
&=\{X^\top(E_{ji}E_{ij}+E_{ji}E_{ij})XX^\top S^\top SX+\\
&X^\top(E_{ji}S+S^\top E_{ij})XX^\top(E_{ji}S+S^\top E_{ij})X+\\
&X^\top(E_{ji}S+S^\top E_{ij})XX^\top(E_{ji}S+S^\top E_{ij})X+\\
&X^\top S^\top SXX^\top(E_{ji}E_{ij}+E_{ji}E_{ij})X-\\
&2z_nX^\top(E_{ji}E_{ij}+E_{ji}E_{ij})X\}\frac{1}{n}\\
&=\{2(X^\top(E_{ji}S+S^\top E_{ij})X)^2\\
&+2X^\top E_{jj}XX^\top S^\top SX+2X^\top S^\top SXX^\top E_{jj}X-4z_nX^\top E_{jj}X\}\frac{1}{n}.
\end{align*}

Using \eqref{partial_2}, we have
\begin{align*}
\partial_\alpha^2[X^\top(S^\top S)^2X]&=\partial_\alpha[
\left\{X^\top[(E_{ji}S+S^\top E_{ij})(S^\top S)+(S^\top S)(E_{ji}S+S^\top E_{ij})]X\right\}]n^{-1/2}\\
&=X^\top [2E_{jj}S^\top S+2(E_{ji}S+S^\top E_{ij})^2+2S^\top SE_{jj}]X\frac{1}{n}.
\end{align*}
By the same arguments, we can show that the traces of the three terms on the right hand side of \eqref{second_order} are bounded above by $n^{-1}$ in magnitude, therefore the second derivative of $f_N$ is bounded by $n^{-2}$. Also by the same reasoning, we can show that there exists some constant $\phi_3(c,\xi,\gamma)$, such that  $|\partial_\alpha^3f_N(s)|\leq\phi_3(c,\xi,\gamma) N^{-5/4}$,  
holds almost surely as $n$ goes to infinity.

We then use similar methods to bound the third derivative of $g_N(s)$.
Define
\begin{align*}
H_n(S)=(X^\top S^\top SX-z_nI_p)^{-1}X^\top X(X^\top S^\top SX-z_nI_p)^{-1}X^\top(S^\top S)^2X,
\end{align*}
then 
$$g_N(s)=\frac{1}{p}\tr[H_n(s)].$$
Note also that
\begin{align*}
(X^\top S^\top SX-z_nI_p)(X^\top X)^{-1}(X^\top S^\top SX-z_nI_p)H_n(S)=X^\top(S^\top S)^2X.
\end{align*}
Taking derivative w.r.t. to $\alpha$ on both sides we have
\begin{align*}
&n^{-1/2}[X^\top(E_{ji}S+S^\top E_{ij})X(X^\top X)^{-1}(X^\top S^\top SX-z_nI_p)H_n(S)+\\
&(X^\top S^\top SX-z_nI_p)(X^\top X)^{-1}X^\top(E_{ji}S+S^\top E_{ij})XH_n(S)]+\\
&(X^\top S^\top SX-z_nI_p)(X^\top X)^{-1}(X^\top S^\top SX-z_nI_p)\partial_\alpha H_n(S)\\
&=n^{-1/2}[X^\top(E_{ji}S+S^\top E_{ij})S^\top SX+X^\top S^\top S(E_{ji}S+S^\top E_{ij})X].
\end{align*}

Using similar techniques, we can show that almost surely $\frac{1}{p}|\tr[\partial_\alpha H_n(S)]|$ is bounded in magnitude by $n^{-3/2}$, $\frac{1}{p}|\tr[\partial_\alpha^2 H_n(S)]|$ is bounded in magnitude by $n^{-2}$, and $\frac{1}{p}|\tr[\partial_\alpha^3 H_n(S)]|$ is bounded in magnitude by $ n^{-5/2}$. Therefore almost surely
$|\partial_\alpha^3g_N(s)|\leq\phi'_3 N^{-5/4},$
for some $\phi'_3=\phi'_3(c,\xi,\gamma)$. Take $\phi=\max(\phi_3,\phi'_3)$, and the proof of Lemma \ref{lemma:bound_third_derivatives} is done.

\end{proof}

\begin{proof} (Proof of Lemma \ref{lemma:bound trace(AB)})
Consider the eigendecompositions of $A,B$,
\begin{align*}
A=Q\left(\begin{array}{ccc}
\lambda_1 &&\\ & \ddots &\\&&\lambda_n\end{array}\right)Q^\top,\,B=P\left(\begin{array}{ccc}
\mu_1 &&\\ & \ddots &\\&&\mu_n\end{array}\right)P^\top,
\end{align*}
then
\begin{align*}
\tr(AB)&=\tr(Q\left(\begin{array}{ccc}
\lambda_1 &&\\ & \ddots &\\&&\lambda_n\end{array}\right)Q^\top P\left(\begin{array}{ccc}
\mu_1 &&\\ & \ddots &\\&&\mu_n\end{array}\right)P^\top).
\end{align*}
Denote the $n$ columns of $Q^\top P$ as $v_1,\ldots,v_n$, which are orthonormal. Then
\begin{align*}
|\tr(AB)|&=|\tr(\left(\begin{array}{ccc}
\lambda_1 &&\\ & \ddots &\\&&\lambda_n\end{array}\right)\sum_{i=1}^n\mu_iv_iv_i^\top)|\\
&=|\sum_{i=1}^n\mu_iv_i^\top\left(\begin{array}{ccc}
\lambda_1 &&\\ & \ddots &\\&&\lambda_n\end{array}\right)v_i|
\leq\sum_{i=1}^n|\mu_i||\lambda|_{\max}(A).
\end{align*}
This finishes the proof.
\end{proof}

\begin{proof} (Proof of Lemma \ref{lemma:uv^top+vu^top})
It is easy to see that $uv^\top+vu^\top$ has rank 2 and
\begin{align*}
(uv^\top+vu^\top)(\frac{u}{\|u\|}+\frac{v}{\|v\|})&=(u^\top v+\|u\|\|v\|)(\frac{u}{\|u\|}+\frac{v}{\|v\|}),\\
(uv^\top+vu^\top)(\frac{u}{\|u\|}-\frac{v}{\|v\|})&=(u^\top v-\|u\|\|v\|)(\frac{u}{\|u\|}-\frac{v}{\|v\|}).
\end{align*}
This finishes the proof.
\end{proof}

\subsubsection{Proof of Lemma \ref{lemma:f_infty,g_infty,limits}}
\label{lemma:f_infty,g_infty,limits:pf}

Let $A=X^\top S^\top SX$ and $B=X^\top S^\top SX-z_nI_n$, and note that we have the relationship
$$A^{-2}-B^{-2}=B^{-1}(B-A)A^{-2}+B^{-2}(B-A)A^{-1}=-z_n(B^{-1}A^{-2}+B^{-2}A^{-1}).$$
Thus
\begin{align*}
f_N(s)-f_\infty(t)&=\frac{1}{p}\tr[(A^{-2}-B^{-2})X^\top(S^\top S)^{2}X]\\
&=-z_n\frac{1}{p}\tr[(B^{-1}A^{-2}+B^{-2}A^{-1})X^\top(S^\top S)^{2}X].
\end{align*}

If the eigenvalues of $A$ are $\lambda_1\geq\ldots\geq\lambda_p>0$, then the eigenvalues of $B$ are $\lambda_1-z_n,\ldots,\lambda_p-z_n$. By Lemma \ref{lemma:bound trace(AB)}, we have
\begin{align*}
\frac{1}{p}|\tr[B^{-1}A^{-2}X^\top(S^\top S)^2X]|&\leq\|A^{-2}X^\top(S^\top S)^2X\|\frac{1}{p}\sum_{i=1}^p\frac{1}{|\lambda_i-z_n|}\\
&\leq\frac{1}{\lambda_p^2}\|X^\top X\|\|S^\top S\|^2\frac{1}{\lambda_p}.
\end{align*}
Recall that $\lambda_p\geq\frac{1}{c^2}(1-\sqrt{\xi})^2$, then we have
\begin{align*}
\frac{1}{p}|\tr[B^{-1}A^{-2}X^\top(S^\top S)^2X]|\leq c^8\frac{(1+\sqrt{\xi})^4}{(1-\sqrt{\xi})^6}.
\end{align*}
By the same argument, we have
\begin{align*}
\frac{1}{p}|\tr[B^{-2}A^{-1}X^\top(S^\top S)^2X]|\leq c^8\frac{(1+\sqrt{\xi})^4}{(1-\sqrt{\xi})^6}.
\end{align*}
Hence
\begin{align*}
|f_N(s)-f_\infty(s)|\leq\frac{1}{p}2c^8\frac{(1+\sqrt{\xi})^4}{(1-\sqrt{\xi})^6}
\end{align*}
holds almost surely. Hence, $ f_N(s)-f_\infty(s)\xrightarrow{a.s.}{}0.$
By the bounded convergence theorem, we have
$\limn|\EE{f_N(s)}-\EE{f_\infty(s)}|=0.$ The other three limit statements can be proved similarly. This finishes the proof.

\subsection{Proof of Theorem \ref{theorem:haar S}}
\label{theorem:haar S:pf}

Suppose that $X$ has the SVD factorization $X=U\Lambda V^\top $ and let $S_1=SU$. The majority of the proof will deal with the following quantities:
\begin{align*}
\tr[(X^\top X)^{-1}]&=\tr(\Lambda^{-2}),\\
\tr[(X^\top S^\top SX)^{-1}]&=\tr[(\Lambda U^\top S^\top SU\Lambda)^{-1}]=\tr[(\Lambda S_1^\top S_1\Lambda)^{-1}],\\
\tr[(X^\top S^\top SX)^{-1}X^\top X]&=\tr[(U^\top S^\top SU)^{-1}]=\tr[(S_1^\top S_1)^{-1}].
\end{align*}
Since we are finding the limits of these quantities, we add the subscript $n$ to matrices like $S_n,U_n$ from now on. Since both $S_n$ and $U_n$ are rectangular orthogonal matrices, we embed them into full orthogonal matrices as
$$\mathbb{S}_n=\left(\begin{array}{c}S_n \\ S_n^\perp\end{array}
\right),\,\mathbb{U}_n=\left(\begin{array}{c}U_n \\ U_n^\perp\end{array}
\right).
$$
Suppose $\frac{1}{p}\Lambda_n S_{1,n}^\top S_{1,n}\Lambda_n$ has an l.s.d. bounded away from zero. Then, the limit of $\frac{1}{p}\tr[(\frac{1}{p}\Lambda_n S_{1,n}^\top S_{1,n}\Lambda_n)^{-1}]$ must equal to the Stieltjes transform of its l.s.d. evaluated at zero. Therefore, we first find the Stieltjes transforms of the l.s.d. of the matrices $\frac{1}{p}\Lambda_n S_{1,n}^\top S_{1,n}\Lambda_n$. The same applies to $\tr[(S_{1,n}^\top S_{1,n})^{-1}]$, except that we replace $\Lambda_n$ with the identity matrix. 

Since $\Lambda_n S_{1,n}^\top S_{1,n}\Lambda_n$ and $S_{1,n}\Lambda_n^{2}S_{1,n}^\top $ have the same non-zero eigenvalues, we first find the l.s.d. of $\frac{1}{n}S_{1,n}\Lambda_n^{2}S_{1,n}^\top $. 
Note that
\begin{align*}
S_{1,n}=S_nU_n&=\left(\begin{array}{cc}I_r & 0\end{array}
\right)
\left(\begin{array}{c}S_n \\ S_n^\perp\end{array}
\right)
\left(\begin{array}{cc} U_n & U_n^\perp\end{array}\right)
\left(\begin{array}{c}I_p \\ 0 \end{array}
\right)\\
&=\left(\begin{array}{cc}I_r & 0\end{array}
\right)\mathbb{S}_n\mathbb{U}_n\left(\begin{array}{c}I_p \\ 0 \end{array}
\right).
\end{align*}
Let $\mathbb{W}_n=\mathbb{S}_n\mathbb{U}_n$, which is again an $n\times n$ Haar-distributed matrix due to the orthogonal invariance of the Haar distribution. Then 
$$S_{1,n}\Lambda_n^2S_{1,n}^\top = \left(\begin{array}{cc}I_r & 0\end{array}
\right)\mathbb{W}_n
\left(\begin{array}{c}I_p \\ 0 \end{array}
\right)\Lambda_n^2
\left(\begin{array}{cc}I_p & 0\end{array}
\right)\mathbb{W}_n^\top \\
\left(\begin{array}{c}I_r \\ 0 \end{array}
\right).
$$
Define
\begin{align}
C_n=\frac{1}{n}\left(\begin{array}{cc}I_r & 0\\0 & 0\end{array}\right)
\mathbb{W}_n
\left(\begin{array}{cc}\Lambda_n^2 & 0\\0&0\end{array}\right)
\mathbb{W}_n^\top 
\left(\begin{array}{cc}I_r & 0\\0 & 0\end{array}\right)=\frac{1}{n}
\left(\begin{array}{cc}S_{1,n}\Lambda_n^2S_{1,n}^\top & 0\\0 & 0\end{array}\right).
\label{def_C_n}
\end{align}
Since $X$ has an l.s.d., we get that the e.s.d. of $\left(\begin{array}{cc} \Lambda_n^2 & 0\\0 & 0\end{array}\right)$ converges to some fixed distribution $F_\Lambda$, and we know that the e.s.d. of $\left(\begin{array}{cc}I_r & 0\\0 & 0\end{array}\right)$ converges to $F_\xi=\xi\delta_1+(1-\xi)\delta_0$. Then according to \cite{Hachem2008anexpression} or Theorem 4.11 of \cite{couillet2011random}, the e.s.d. of $C_n$ converges to a distribution $F_C$, whose $\eta$-transform $\eta_C$ is the unique solution of the following system of equations, defined for all $z\in \mathbb{C}^+$: 
\begin{align*}
\eta_C(z)&=\int\frac{1}{z\gamma(z)t+1}dF_\xi(t)=\frac{\xi}{z\gamma(z)+1}+(1-\xi),\\
\gamma(z)&=\int\frac{t}{\eta_C(z)+z\delta(z)t}dF_\Lambda(t),\\
\delta(z)&=\int\frac{t}{z\gamma(z)t+1}dF_\xi(t)=\frac{\xi}{z\gamma(z)+1}.
\end{align*}
Moreover, we note that if the support of $F_\Lambda$ outside of the point mass at zero is bounded away from the origin, then the same is also true for $F_C$. Indeed, this follows directly from the form of $\Lambda_n S_{1,n}^\top S_{1,n}\Lambda_n$, as its smallest eigenvalue can be bounded below as
$$\lambda_{\min}(\Lambda_n S_{1,n}^\top S_{1,n}\Lambda_n) 
\ge
\lambda_{\min}(\Lambda_n)^2 \lambda_{\min}(S_{1,n}^\top S_{1,n}).
$$
Moreover, by assumption $\lambda_{\min}(\Lambda_n)>c>0$ for some universal constant $c$, and clearly $\lambda_{\min}(S_{1,n}^\top S_{1,n})=1$, as $S_{1,n}$ is a partial orthogonal matrix. This ensures that we can use the Stieltjes transform as a tool to calculate the limiting traces of the inverse.

Returning to our equations, using the first and the third equations to solve for $\delta(z)$ and $\gamma(z)$ in terms of $\eta_C(z)$, substituting them in the second equation, we get the following fixed point equation
\begin{align}
\label{etafixpoint}
\eta_C(z)=\eta_\Lambda(z(1+\frac{\xi-1}{\eta_{C}(z)}).
\end{align}
According to the definition of $\eta$-transform \eqref{def_eta_transform}, for any distribution $F$ with a point mass $f_F(0)$ at zero, we have
\begin{align*}
\eta_F(z)=\int_{t\neq0}\frac{1}{1+zt}dF(t)+f_F(0).
\end{align*}
Note that $f_C(0)=f_\Lambda(0)=1-\gamma$. Since the l.s.d. of $X$ is compactly supported and bounded away from the origin, we know $\inf[supp(f_\Lambda)\cap\R^*]$ and $\inf[supp(f_\Lambda)\cap\R^*]$ are greater than zero, thus $\frac{1}{t}$ is integrable on the set $\{t>0\}$ w.r.t. $F_\Lambda$ and $F_C$. Since $|\frac{z}{1+tz}|<\frac{1}{t}$ when $z>0,t>0$, by the dominated convergence theorem we have
\begin{align*}
\lim_{z\rightarrow\infty}\int_{t\neq0}\frac{z}{1+tz}dF_C(t)&=\int_{t\neq0}\frac{1}{t}dF_C(t),\\
\lim_{z\rightarrow\infty}\int_{t\neq0}\frac{z}{1+tz}dF_\Lambda(t)&=\int_{t\neq0}\frac{1}{t}dF_\Lambda(t),
\end{align*}
and hence
\begin{align*}
\int_{t\neq0}\frac{1}{t}dF_C(t)&=\lim_{z\rightarrow\infty}z(\eta_C(z)-(1-\gamma)),\numberthis\label{equation:lim_z_eta_1}\\
\int_{t\neq0}\frac{1}{t}dF_\Lambda(t)&=\lim_{z\rightarrow\infty}z(\eta_\Lambda(z)-(1-\gamma)),\numberthis\label{equation:lim_z_eta_2}
\end{align*}
and
\begin{align*}
\lim_{z\rightarrow\infty}\eta_C(z)&=\lim_{z\rightarrow\infty}\int_{t\neq0}\frac{1}{1+zt}dF_C(t)+(1-\gamma)\\
&=\int_{t\neq0}\lim_{z\rightarrow\infty}\frac{1}{1+zt}dF_C(t)+(1-\gamma)\\
&=1-\gamma.\numberthis\label{equation:lim_z_eta_3}
\end{align*}

Subtracting $1-\gamma$ from both sides of \eqref{etafixpoint}, multiplying by $z(1+\frac{\xi-1}{\eta_C(z)})$, letting $z\rightarrow\infty$, we obtain
\begin{align*}
\lim_{z\rightarrow\infty}z(1+\frac{\xi-1}{\eta_C(z)})[\eta_C(z)-(1-\gamma)]&=\lim_{z\rightarrow\infty}z(1+\frac{\xi-1}{\eta_C(z)})[\eta_\Lambda(z(1+\frac{\xi-1}{\eta_C(z)}))-(1-\gamma)].
\end{align*}
Note that RHS equals $\int_{t\neq0}\frac{1}{t}dF_\Lambda(t)$ by \eqref{equation:lim_z_eta_2}, and
\begin{align*}
LHS&=\lim_{z\rightarrow\infty}z(1+\frac{\xi-1}{\eta_C(z)})[\eta_C(z)-(1-\gamma)]\\
&=\lim_{z\rightarrow\infty}z[\eta_C(z)-(1-\gamma)](1+\frac{\xi-1}{1-\gamma})\\
&=\int_{t\neq0}\frac{1}{t}dF_C(t)\frac{\xi-\gamma}{1-\gamma},
\end{align*}
where the second and the third equations follow from \eqref{equation:lim_z_eta_3} and \eqref{equation:lim_z_eta_2}.
This shows that
\begin{align*}
\int_{t\neq0}\frac{1}{t}dF_\Lambda(t)=\frac{\xi-\gamma}{1-\gamma}\int_{t\neq0}\frac{1}{t}dF_C(t),
\end{align*}
therefore we have proved that as $n\rightarrow\infty$,
\begin{align*}
\frac{\tr[(\Lambda S_1^\top S_1^\top \Lambda)^{-1}]}{\tr(\Lambda^{-2})}\rightarrow
\frac{\int_{t\neq0}\frac{1}{t}dF_C(t)}{\int_{t\neq0}\frac{1}{t}dF_\Lambda(t),}=
\frac{1-\gamma}{\xi-\gamma},
\end{align*}
thus
$$\limn VE(\hbeta_s,\hbeta)=\frac{1-\gamma}{\xi-\gamma}.$$
This finishes the evaluation of $VE$.

Next, to evaluate of $PE$, we argue as follows: 
In the definition of $C_n$ in \eqref{def_C_n}, replace $\Lambda_n$ by the identity matrix. Since the results do not depend the l.s.d. of $\Lambda_n$, it follows directly that
\begin{align*}
PE
=\frac{\tr[(X^\top S^\top SX)^{-1}X^\top X]}{p}
=\frac{\tr[(S_1^\top S_1)^{-1}]}{\tr(I_p)}
\rightarrow\frac{1-\gamma}{\xi-\gamma}.
\end{align*}
Next, to evaluate the limit of $OE$, we use the additional assumption on $X$, that is, $X=Z\Sigma^{1/2}$, where $Z$ has iid entries of zero mean, unit variance and finite fourth moment.

Note that (with convergence below always meaning almost sure convergence)
\begin{align*}
\EE{x_t^\top(X^\top X)^{-1}x_t)}\rightarrow\frac{\gamma}{1-\gamma},
\end{align*}
which has been proved in Section \ref{gsdx:pf}, and 
$$1+\EE{x_t^\top   (X^\top    X)^{-1}x_t}\rightarrow1+\frac{\gamma}{1-\gamma}=\frac{1}{1-\gamma}.$$
On the other hand,
\begin{align*}
&\EE{x_t^\top    (X^\top    S^\top    SX)^{-1}x_t}=\tr(\EE{X^\top    S^\top    SX}^{-1}\EE{x_tx_t^\top    })\\
&=\tr(\EE{(\Sigma^{1/2}Z^\top    S^\top    SZ\Sigma^{1/2})^{-1}}\Sigma)
=\tr(\EE{Z^\top    S^\top    SZ}^{-1}).
\end{align*}
Define $C_n=\frac{1}{n}Z^\top    S^\top    SZ$, then the e.s.d. of $C_n$ converges to a distribution $F_C$, whose Stieltjes transform $m(z)=m_C(z),z\in\mathbb{C}^+$ is given by \citep{bai2009spectral}
$$m(z)=\frac{1}{\int\frac{s}{1+\gamma se}dF_{S^\top    S}(s)-z}=\frac{1}{\frac{\xi}{1+\gamma e}-z},$$
where 
$$e=\frac{1}{\int\frac{s}{1+\gamma se}dF_{S^\top    S}(s)-z}=\frac{1}{\frac{\xi}{1+\gamma e}-z}.$$
And here $F_{S^\top S}$ is the l.s.d. of $S^\top S$, which is $\xi\delta_1+(1-\xi)\delta_0$.
Solving these equations gives
$$m(z)=e(z)=\frac{\xi-\gamma-z+\sqrt{(\xi-\gamma-z)^2-4z\gamma}}{2z\gamma}.$$
Therefore 
\begin{align*}
\lim_{z\rightarrow0}m(z)&=\frac{-1-\frac{2(\gamma-\xi)-4\gamma}{2(\xi-\gamma)}}{2\gamma}
=\frac{-1+\frac{\xi+\gamma}{\xi-\gamma}}{2\gamma}
=\frac{1}{\xi-\gamma}.
\end{align*}
Thus
\begin{align*}
\tr((Z^\top    S^\top    SZ)^{-1})=\frac{1}{n}\tr((\frac{1}{n}Z^\top    S^\top    SZ)^{-1})\xrightarrow{a.s.}\gamma m_C(0)=\frac{\gamma}{\xi-\gamma}.
\end{align*}
Therefore
$$1+ \EE{x_t^\top    (X^\top    S^\top    SX)^{-1} x_t}\rightarrow1+\frac{\gamma}{\xi-\gamma}=\frac{1}{1-\gamma/\xi},$$
and we have proved
$$\limn OE(\hbeta_s,\hbeta)
=\limn\frac{1+ \EE{x_t^\top    (X^\top    S^\top    SX)^{-1} x_t}}{1+ \EE{x_t^\top   (X^\top    X)^{-1}x_t}}=\frac{1-\gamma}{1-\gamma/\xi}.$$


This finishes the proof.
\subsubsection{Checking the free multiplicative convolution property}
\label{proof:check multiplicative conv}
Recall that the $S$-transform of a distribution $F$ is defined as the solution to the equation
\begin{equation*}
m_F(\frac{z+1}{zS(z)})=-zS(z).
\end{equation*}

For more references, see for instance \cite{voiculescu1992free,hiai2006semicircle,nica2006lectures,anderson2010introduction}.

Since $m(\frac{z+1}{zS(z)})=-zS(z),\eta(z)=\frac{1}{z}m(-\frac{1}{z})$, we have
\begin{align*}
-zS(z)=m(\frac{z+1}{zS(z)})=-\frac{zS(z)}{z+1}\eta(-\frac{zS(z)}{z+1}),
\end{align*}
where $S(z)$ is the $S$-transform. Therefore
\begin{align*}
\eta_\Lambda(-\frac{zS_\Lambda(z)}{z+1})=z+1, \,
\eta_C(-\frac{zS_C(z)}{z+1})=z+1.
\end{align*}
Let $x=-\frac{z}{z+1}S_C(z)$, then $\eta_C(x)=z+1$ and \eqref{etafixpoint} gives
\begin{align*}
z+1&=\eta_C(x)
=\eta_\Lambda(x(1+\frac{\xi-1}{\eta_C(x)}))
=\eta_\Lambda(-\frac{z}{z+1}S_C(z)(1+\frac{\xi-1}{z+1}))\\
&=\eta_\Lambda(-\frac{z}{z+1}S_C(z)\frac{z+\xi}{z+1})
=\eta_\Lambda(-\frac{z}{z+1}S_\Lambda(z)).
\end{align*}
Therefore
$S_\Lambda=\frac{z+\xi}{z+1}S_C(z),$
and equivalently
$S_C(z)=S_\Lambda(z)\frac{z+1}{z+\xi}.$
Let $S_0(z)=\frac{z+1}{z+\xi}$ be the $S$-transform of some distribution $F_0$, then the corresponding Stieltjes transform is $m_0(z)=\frac{\xi}{1-z}+\frac{1-\xi}{-z}$, which is the Stieltjes transform for $F_0=\xi\delta_1+(1-\xi)\delta_0$. This shows that $F_C$ is a freely multiplicative convolution of $F_\Lambda$ and $\xi\delta_1+(1-\xi)\delta_0$. 

\subsection{Proof of Theorem  \ref{theorem:Hadamard S}}
\label{theorem:Hadamard S:pf}

Note that $B,H$ and $D$ are all symmetric matrices satisfying
\begin{align*}
B^2=B,\,H^2=I_n,\,D^2=I_n,
\end{align*}
and $P$ is also an orthogonal matrix,
therefore
\begin{align*}
S^\top S&=P^\top DHBHDP\\
(S^\top S)^2&=P^\top DHBHDPP^\top DHBHDP\\
&=P^\top DHBHDP=S^\top S.
\end{align*}
By Proposition \ref{proposition:finite n calculation for orthogonal S}, we only need to find
\begin{align*}
\tr[(X^\top S^\top SX)^{-1}]
=\tr[(X^\top P^\top DHBHDPX)^{-1}],\numberthis\label{equation:hadamard trace 1}
\end{align*}
and 
\begin{align*}
\tr[(X^\top S^\top SX)^{-1}X^\top X]
=\tr[(X^\top P^\top DHBHDPX)^{-1}X^\top X].\numberthis\label{equation:hadamard trace 2}
\end{align*}
We first have the following observation.
\begin{lemma} For a uniformly distributed permutation matrix $P$, diagonal matrix $B$ with iid diagonal entries of distribution $\mu_B=\frac{r}{n}\delta_1+(1-\frac{r}{n})\delta_0$, diagonal matrix $D$ with iid sign random variables, equal to $\pm1$ with probability one half, and Hadamard matrix $H$, we have the following equation in distribution
$$X^\top (P^\top DH)B(HDP)X\stackrel{d}{=}X^\top (P^\top DHDP)B(P^\top DHDP)X.$$
\label{lemma:Hadamard equal in distribution}
\end{lemma}
This is true, because we are simply permuting the diagonal matrix of iid Bernoullis in the middle term; but see the end of this section for a formal proof. We call $DP$ the signed-permutation matrix and $W=P^\top DHDP$ the bi-signed-permutation Hadamard matrix. Thus by equations \eqref{equation:hadamard trace 1}, \eqref{equation:hadamard trace 2}, and Lemma \ref{lemma:Hadamard equal in distribution},
\begin{align*}
\EE{\tr[(X^\top S^\top SX)^{-1}]}&=\EE{\tr[(X^\top (P^\top DHDP)B(P^\top DHDP)X)^{-1}]}\\
&=\EE{\tr[(X^\top WBWX)^{-1}]},\\
\EE{\tr[(X^\top S^\top SX)^{-1}X^\top X]}&=\EE{\tr[(X^\top (P^\top DHDP)B(P^\top DHDP)X)^{-1}X^\top X]}\\
&=\EE{\tr[(X^\top WBWX)^{-1}X^\top X]}.
\end{align*}
Since $X^\top WBWX$ has the same nonzero eigenvalues as $BWXX^\top WB$, we first find the l.s.d. of 
\begin{align*}
C_n=\frac{1}{n}B_nW_nX_nX_n^\top W_nB_n.
\end{align*}
The following lemma states the asymptotic freeness regarding Hadamard matrix, which will be used to find the l.s.d. of $C_n$. For more references on free probability, see for instance \cite{voiculescu1992free,hiai2006semicircle,nica2006lectures,anderson2010introduction}.
\begin{lemma}(Freeness of bi-signed-permutation Hadamard matrix) 
Let $X_n,B_n,W_n$ be defined above, that is, $X_n$ is an $n\times n$ deterministic matrix with uniformly bounded spectral norm and has l.s.d. $\mu_X$, $B_n$ is a diagonal matrix with iid diagonal entries, and $W_n$ is a bi-signed-permutation matrix. Then
\begin{align*}
\{B_n,\frac{1}{n}W_nX_nX_n^\top W_n\}
\end{align*}
are asymptotically free in the limit of the non-commutative probability spaces of random matrices, as described in Section \ref{sec:rmt}. The law of
\begin{align*}
C_n=\frac{1}{n}B_nW_nX_nX_n^\top W_nB_n
\end{align*}
converges to the freely multiplicative convolution of $\mu_B$ and $\mu_X$, that is, $C_n$ has l.s.d. $\mu_C=\mu_B\boxtimes\mu_X$.
\label{lemma:freeness of hadamard matrix}
\end{lemma}

This follows directly from Corollaries 3.5, 3.7 of \cite{anderson2014asymptotically}. See also Lemma 1 of \cite{tulino2010capacity} for earlier results on the Fourier transform.

We use $\mu_B$ and $\mu_X$ to denote the elements in the limiting non-commutative probability space, their laws, and their corresponding probability distributions interchangeably. Since $\mu_B=\xi\delta_1+(1-\xi)\delta_0$, we have $S_{\mu_B}=\frac{z+1}{z+\xi}$. From the aymptotic freeness, it follows that the $S$-transform of $\mu_C$ is the product of that of $\mu_B,\mu_X$, so that
$$S_{\mu_C}(z)=S_{\mu_X}(z)S_{\mu_B}(z)=S_{\mu_X}(z)\frac{z+1}{z+\xi}.$$
We will now simplify this relation. First, note that by the definition of the S-transform, we have 
$$\eta_{\mu_C}(-\frac{z}{z+1}S_{\mu_C}(z))=z+1.$$
Letting $y=-\frac{z}{z+1}S_{\mu_C}$, we have $\eta_{\mu_C}(y)=z+1$. In addition, we can simplify the original relation as
\begin{align*}
S_{\mu_X}&=\frac{z+\xi}{z+1}S_{\mu_C}(z)=-\frac{z+\xi}{z}y,\\
z+1&=\eta_{\mu_X}(-\frac{z}{z+1}S_{\mu_X}(z))
=\eta_{\mu_X}(\frac{z+\xi}{z+1}y)\\
&=\eta_{\mu_X}((1+\frac{\xi-1}{z+1})y)
=\eta_{\mu_X}((1+\frac{\xi-1}{\eta_{\mu_C}(y)})y)
=\eta_{\mu_C}(z).
\end{align*}
So we have obtained
\begin{align*}
\eta_{\mu_X}((1+\frac{\xi-1}{\eta_{\mu_C}(y)})y)=\eta_{\mu_C}(y).
\end{align*}
This is the same equation as what we obtained in \eqref{etafixpoint} in the proof of Haar projection. Therefore as $n\rightarrow\infty$, we have as required
\begin{align*}
\limn VE(\hbeta_s,\hbeta)=\frac{1-\gamma}{\xi-\gamma}.
\end{align*}
Next we consider
$$\EE{\tr[(X^\top WBWX)^{-1}X^\top X]}.$$
Since $X$ has the SVD $X=U\Lambda V^\top$, we have
$$\EE{\tr[(X^\top WBWX)^{-1}X^\top X]}=\EE{\tr[(U^\top WBWU)^{-1}]}.$$
Thus we can repeat the above reasoning, except that we replace $X$ by $U$. Since the result does not depend on $X$, we have
\begin{align*}
\limn PE(\hbeta_s,\hbeta)&=\limn\frac{\EE{\tr[(X^\top S^\top SX)^{-1}X^\top X]}}{p}\\
&=\limn\frac{\EE{\tr[(U^\top WBWU)^{-1}]}}{\tr[U^\top U]}\\
&=\limn VE(\hbeta_s,\hbeta)
=\frac{1-\gamma}{\xi-\gamma}.
\end{align*}

For $OE$, since $S$ satisfies $(S^\top S)^2=S^\top S$ and the e.s.d. of $S^\top S$ converges to $\xi\delta_1+(1-\xi)\delta_0$, the same reasoning as in Theorem \ref{theorem:haar S} also holds in this case for Hadamard projection. This finishes the proof. 



\begin{proof}(Proof of Lemma \ref{lemma:Hadamard equal in distribution})
Note that both $B$ and $D$ are diagonal matrices whose diagonal entries are iid random variables, and $P$ is a permutation matrix. Define $\tilde{B}=P BP^\top$ and $\tilde{D}=P^\top DP$, then
we have
$$\tilde{B}\stackrel{d}{=}B,\quad\tilde{D}\stackrel{d}{=}D$$
and
\begin{align}
DP=P\tilde{D},\quad P^\top D=\tilde{D}P^\top.
\label{equation:DP=PD}
\end{align}
Hence
\begin{align*}
X^\top P^\top DHDPBP^\top DHDPX&=X^\top P^\top DHP\tilde{D}B\tilde{D}P^\top HDPX\\
&=X^\top P^\top DHPB\tilde{D}^2P^\top HDPX\\
&=X^\top P^\top DHPBP^\top HDPX\\
&=X^\top P^\top DH\tilde{B} HDPX\\
&\stackrel{d}{=}X^\top P^\top DHB HDPX,
\end{align*}
where the first equation follows from \eqref{equation:DP=PD}, the second equation holds because $\tilde{D}$ and $B$ are diagonal entries so they commute, while the third equation holds because $\tilde{D}^2=I_n$.
\end{proof}

\subsection{Proof of Theorem \ref{theorem:uniform sampling}}
\label{theorem:uniform sampling:pf}

We can take 
$$S=\left(\begin{array}{ccc}s_1 & 0 & 0\\
0 & \ddots & 0\\
0 & 0 & s_n\end{array}
\right),$$
which is an $n\times n$ diagonal matrix and $s_i$-s are iid random variables with $\PP{s_i=1}=\frac{r}{n}$ and $\PP{s_i=0}=1-\frac{r}{n}$. Since $s_i^2=s_i$, we have $S^2=S$, hence
\begin{align*}
VE(\hbeta_s,\hbeta)&=\frac{\EE{\tr[(X^\top SX)^{-1}]}}{\tr[(X^\top X)^{-1}]}\\
PE(\hbeta_s,\hbeta)&=\frac{\EE{\tr[(X^\top SX)^{-1}X^\top X]}}{p}.
\end{align*}
Since $X$ is unitarily invariant and $S$ is a diagonal matrix independent from $X$, $\{S,X,X^\top\}$ are almost surely asymptotically free in the non-commutative probability space by Theorem 4.3.11 of \cite{hiai2006semicircle}. Since the law of $S$ converges to $\mu_S=\xi\delta_1+(1-\xi)\delta_0$, the law of $X$ converges to $\mu_X$, thus the law of $SXX^\top S$ converges to the freely multiplicative convolution $\mu_S\boxtimes\mu_X$. The rest of the proof is the same as that in the proof of Theorem \ref{theorem:Hadamard S}.

\subsection{Proof of Theorem \ref{theorem:sampling elliptical model}}
\label{theorem:sampling elliptical model:pf}
Define 
\begin{align*}
S=\left(\begin{array}{ccc}
s_1&&\\&\ddots&\\&&s_n\end{array}
\right),\quad W=\left(\begin{array}{ccc}
w_1&&\\&\ddots&\\&&w_n\end{array}
\right),
\end{align*}
where the $s_i$-s are independent and $s_i|\pi_i\sim Bernoulli(\pi_i)$. $S$ is independent of $Z$ because $\pi_i$ is independent of $z_i$, by the assumption. $W$ has l.s.d. $F_w$. According to Proposition \ref{proposition:finite n calculation}, the values of $VE$, $PE$ are determined by $\tr[(X^\top X)^{-1}]$, $\tr[Q_1(S,X)]=\tr[(X^\top SX)^{-1}]$, and  $\tr[Q_2(S,X)]$.
Note that under the elliptical model $X=WZ\Sigma^{1/2}$, we have
\begin{align*}
\tr[(X^\top X)^{-1}]&=\tr[(\Sigma^{1/2}Z^\top W^2 Z\Sigma^{1/2})^{-1}],\\
\tr[Q_1(S,X)]&=\tr[(\Sigma^{1/2}Z^\top WSW Z\Sigma^{1/2})^{-1}],\\
\tr[Q_2(S,X)]&=\tr[(Z^\top WSWZ)^{-1}Z^\top W^2Z].
\end{align*}

Note that the e.s.d. of $\Sigma$ converges in distribution to some probability distribution $F_\Sigma$, and the e.s.d. of $WSW$ converges in distribution to $F_{sw^2}$, the limiting distribution of $s_iw_i^2$, $i=1,\ldots,n$. Again from the results of \cite{lixin2007spectral} or \cite{paul2009no}, with probability 1, the e.s.d. of $C_n=\frac{1}{n}\Sigma^{1/2}Z^\top WSW Z\Sigma^{1/2}$ converges to a probability distribution function $F_C$, whose Stieltjes transform $m_C(z)$, for $z\in\mathbb{C}^+$ is given by
\begin{align*}
m_C(z)=\int\frac{1}{t\int\frac{u}{1+\gamma e_Cu}dF_{sw^2}(u)-z}dF_\Sigma(t),
\end{align*}
where $e_C=e_C(z)$ is the unique solution in $\mathbb{C}^+$ of the equation
\begin{align*}
e_C=\int\frac{t}{t\int\frac{u}{1+\gamma e_Cu}dF_{sw^2}(u)-z}dF_\Sigma(t).
\end{align*}
Similarly, the e.s.d. of $D_n=\frac{1}{n}\Sigma^{1/2}Z^\top W^2 Z\Sigma^{1/2}$ converges to a probability distribution $F_D$, whose Stieltjes transform $m_D(s)$, for $z\in\mathbb{C}^+$ is given by
\begin{align*}
m_D(z)=\int\frac{1}{t\int\frac{u}{1+\gamma e_Du}dF_{w^2}(u)-z}dF_\Sigma(t),
\end{align*}
where $e_D=e_D(z)$ is the unique solution in $\mathbb{C}^+$ of the equation
\begin{align*}
e_D=\int\frac{t}{t\int\frac{u}{1+\gamma e_Du}dF_{w^2}(u)-z}dF_\Sigma(t).
\end{align*}
Since $F_C$ and $F_D$ have no point mass at the origin, we can set $z=0$ \cite{couillet2014analysis}. Therefore
\begin{align*}
m_C(0)&=\frac{1}{\int\frac{u}{1+\gamma e_C(0)u}dF_{sw^2}(u)}\int\frac{1}{t}dF_\Sigma(t),\quad
e_C(0)=\frac{1}{\int\frac{u}{1+\gamma e_C(0)u}dF_{sw^2}(u)}.
\end{align*}
Note also that
\begin{align*}
e_C(0)=\frac{\gamma e_C(0)}{\int\frac{\gamma e_C(0)u}{1+\gamma e_C(0)u}dF_{sw^2}(u)}
=\frac{\gamma e_C(0)}{1-\eta_{sw^2}(\gamma e_C(0))},
\end{align*}
thus $\eta_{sw^2}(\gamma e_C(0))=1-\gamma$,
and
\begin{align}
m_C(0)=e_C(0)\int\frac{1}{t}dF_\Sigma(t)=\frac{\eta_{sw^2}^{-1}(1-\gamma)}{\gamma}\int\frac{1}{t}dF_\Sigma(t).
\label{m_C(0)}
\end{align}
Similarly,
\begin{align*}
m_D(0)=e_D(0)\int\frac{1}{t}dF_\Sigma(t)=\frac{\eta_{w^2}^{-1}(1-\gamma)}{\gamma}\int\frac{1}{t}dF_\Sigma(t).
\end{align*}
Hence, again by the same argument as we have seen several times before, the traces have limits that can be evaluated in terms of Stieltjes transforms, and we have
\begin{align*}
VE(\hbeta_s, \hbeta)&=\frac{\tr[Q_1(S,X)]}{\tr[(X^\top X)^{-1}]}
=\frac{\tr[(\Sigma^{1/2}Z^\top WSW Z\Sigma^{1/2})^{-1}]}{\tr[(\Sigma^{1/2}Z^\top W^2 Z\Sigma^{1/2})^{-1}]}\\
&\rightarrow\frac{m_C(0)}{m_D(0)}
=\frac{\eta_{sw^2}^{-1}(1-\gamma)}{\eta_{w^2}^{-1}(1-\gamma)},
\end{align*}
and the result for $VE$ follows. 

We then deal with $PE$. Note that
\begin{align*}
PE(\hbeta_s, \hbeta)&=\frac{\EE{\tr[(Z^\top WSWZ)^{-1}Z^\top W^2Z]}}{p}.
\end{align*}
We first assume that $Z$ has iid $\N(0, 1)$ entries. Denote $T_1=WSW$, $T_2=W(I-S)W$. Since $S$ is a diagonal matrix whose diagonal entries are 1 or 0, W is also a diagonal matrix, $T_1$ and $T_2$ are both diagonal matrices and the set of their nonzero entries is complementary. So $Z^\top T_1Z$ and $Z^\top T_2Z$ are independent from each other and $T_1+T_2=W^2$. We have
\begin{align*}
\EE{\tr[(Z^\top WSWZ)^{-1}Z^\top W^2Z]}&=\EE{\tr[(Z^\top T_1Z)^{-1}Z^\top(T_1+T_2)Z]}\\
&=\EE{\tr[I_p+(Z^\top T_1Z)^{-1}Z^\top T_2Z]}\\
&=p+\tr[\EE{(Z^\top T_1Z)^{-1}}\EE{Z^\top T_2Z}].
\end{align*}
Note that
\begin{align*}
\EE{(Z^\top T_2Z)_{ij}}&=\sum_{k=1}^n\EE{z_{ki}T_{2,kk}z_{kj}}
=\sum_{k=1}^nT_{2,kk}\delta_{ij}
\end{align*}
thus
\begin{align*}
\EE{Z^\top T_2Z}&=\EE{\tr(T_2)}I_p,\\
\EE{\tr[(Z^\top WSWZ)^{-1}Z^\top W^2Z]}&=p+\EE{\tr(T_2)}\tr[\EE{(Z^\top T_1Z)^{-1}}].
\end{align*}
Note that $\frac{1}{n}Z^\top WSWZ$ is equal to $C_n$ with $\Sigma$ replaced by the identity. Thus by \eqref{m_C(0)},
\begin{align*}
\frac{1}{p}\tr[(\frac{1}{n}Z^\top WSWZ)^{-1}]]&\xrightarrow{a.s.}\frac{\eta_{sw^2}^{-1}(1-\gamma)}{\gamma},\\
\tr[(Z^\top WSWZ)^{-1}]]&\xrightarrow{a.s.}\eta_{sw^2}^{-1}(1-\gamma),
\end{align*} 
thus
\begin{align*}
\limn PE(\hbeta_s,\hbeta)&=1+\frac{1}{p}\tr(T_2)\eta_{sw^2}^{-1}(1-\gamma)\\
&=1+\frac{1}{\gamma}\EE{w^2(1-s)}\eta_{sw^2}^{-1}(1-\gamma)
\end{align*}
Then we use a similar Lindeberg swapping argument as in Theorem \ref{theorem:haar S} to show extend this to $Z$ with iid entries of zero mean, unit variance and finite fourth moment. This finishes the proof for $PE$. For the last claim, for $OE$, note that
\begin{align*}
\EE{x_t^\top(X^\top X)x_t}&=\EE{w^2}\EE{z_t^\top(Z^\top W^2Z)^{-1}z_t}\\
&=\EE{w^2}\EE{\tr[(Z^\top W^2Z)^{-1}]}\\
&\rightarrow\EE{w^2}\eta_{w^2}^{-1}(1-\gamma),
\end{align*}
and that
\begin{align*}
\EE{x_t^\top(X^\top S^\top S X)x_t}&=\EE{w^2}\EE{z_t^\top(Z^\top WSWZ)^{-1}z_t}\\
&=\EE{w^2}\EE{\tr[(Z^\top WSWZ)^{-1}]}\\
&\rightarrow\EE{w^2}\eta_{sw^2}^{-1}(1-\gamma).
\end{align*}
Thus 
\begin{align*}
\limn OE(\hbeta_s,\hbeta)=\limn\frac{1+\EE{x_t^\top(X^\top S^\top SX)^{-1}x_t}}{1+\EE{x_t^\top(X^\top X)^{-1}x_t}}=\frac{1+\EE{w^2}\eta_{sw^2}^{-1}(1-\gamma)}{1+\EE{w^2}\eta_{w^2}^{-1}(1-\gamma)},
\end{align*}

This finishes the proof.

\subsubsection*{Proof of leverage sampling}
\label{theorem:leverage sampling:pf}
It suffices to show that leverage score sampling that samples the $i$-th row with probability $\min(\frac{r}{p}h_{ii},1)$ is equivalent to sample with probability $\min\left[\frac{r}{p}\left(1-\frac{1}{1+w^2\eta_{w^2}^{-1}(1-\gamma)}\right), 1\right]$. Given that the latter probability is independent from $z_i$, the statement of the corollary will then follow directly from Theorem \ref{theorem:sampling elliptical model}.

To see this equivalence, first note that
\begin{align*}
h_{ii}&=x_i^\top(\sum_{j\neq i}x_jx_j^\top+x_ix_i^\top)^{-1}x_i=x_i^\top(\sum_{j\neq i}x_jx_j^\top)^{-1}x_i-\frac{(x_i^\top(\sum_{j\neq i}x_jx_j^\top)^{-1}x_i)^2}{1+x_i^\top(\sum_{j\neq i}x_jx_j^\top)^{-1}x_i}\\
&=\frac{x_i^\top(\sum_{j\neq i}x_jx_j^\top)^{-1}x_i}{1+x_i^\top(\sum_{j\neq i}x_jx_j^\top)^{-1}x_i}
\end{align*}
and
\begin{align*}
\frac{1}{1-h_{ii}}
&=1+x_i^\top(\sum_{j\neq i}x_jx_j^\top)^{-1}x_i
=1+w_i^2z_i^\top\Sigma^{1/2}(\sum_{j\neq i}x_jx_j^\top)^{-1}\Sigma^{1/2}z_i\\
&=1+w_i^2z_i^\top(\sum_{j\neq i}w_j^2z_jz_j^\top)^{-1}z_i.
\end{align*}
Denote $R=\sum_{j=1}^nw_j^2z_jz_j^\top ,\, R_{(i)}=\sum_{j\neq i}w_j^2z_jz_j^\top,$ so that $\frac{1}{1-h_{ii}}=1+w_i^2z_i^\top R_{(i)}^{-1}z_i.$

Since $z_i$ and $R_{(i)}$ are independent for each $i=1,\ldots,n$, while $z_i$ has iid entries of zero mean and unit variance and bounded moments of sufficienctly high order, then by the concentration of quadratic forms lemma \ref{quad_form} cited below, we have
\begin{align*}
\frac{1}{n}z_i^\top R_{(i)}^{-1}z_i-\frac{1}{n}\tr(R_{(i)}^{-1})\xrightarrow{a.s.}{}0.
\end{align*}
\begin{lemma}[Concentration of quadratic forms, consequence of Lemma B.26 in \citet{bai2009spectral}] Let $x \in \RR^p$ be a random vector with iid entries and $\EE{x} = 0$, for which $\EE{(\sqrt{p}x_i)^2} = \sigma^2$ and $\sup_i \EE{(\sqrt{p}x_i)^{4+\eta}}$ $ < C$ for some $\eta>0$ and $C <\infty$. Moreover, let $A_p$ be a sequence of random $p \times p$ symmetric matrices independent of $x$, with uniformly bounded eigenvalues. Then the quadratic forms $x^\top A_p x $ concentrate around their means at the following rate
$$P(|x^\top A_p x - p^{-1} \sigma^2 \tr A_p|^{2+\eta/2}>C) \le C p^{-(1+\eta/4)}.$$
\label{quad_form}
\end{lemma}
To use lemma \ref{quad_form}, we only need to guarantee that the smallest eigenvalue of $R_{(i)}$ is uniformly bounded below. For this, it is enough that the smallest eigenvalue of $R$ is uniformly bounded below. Since $w_i$ are bounded away from zero, this property follows from the corresponding one for the sample covariance matrix of $z_i$, which is just the well-known Bai-Yin law \citep{bai2009spectral}. 

Continuing with our argument, by the standard rank-one-perturbation argument \citep{bai2009spectral}, we have $\limn\frac{1}{n}\tr[R_{(i)}^{-1}]-\frac{1}{n}\tr[R^{-1}]=0$, since $R_{(i)}$ is a rank-one perturbation of $R$. 
Recall that $Z$ has iid entries satisfying $\EE{Z_{ij}}=0$,$\EE{Z_{ij}^2}=1$. Moreover, it is easy to see that by the $4+\eta$-th moment assumption we have for each $\delta>0$ that
\begin{align*}
\frac{1}{\delta^2np}\sum_{i,j}\EE{Z_{ij}^2I_{[|Z_{ij}|>\delta\sqrt{n}}]}\rightarrow0,\,\text{ as }n\rightarrow\infty.
\end{align*}
Also, the e.s.d. of $W^2$ converges weakly to the distribution of $w^2$. By the results of \cite{lixin2007spectral} or \cite{paul2009no}, with probability 1, the e.s.d. of 
$B_n=n^{-1}Z^\top W^2Z$
converges in distribution to a probability distribution $F_B$ whose Stieltjes transform satisfies
\begin{align*}
m_B(z)=\frac{1}{\int\frac{s}{1+\gamma e_Bs}dF_{w^2}(s)-z},
\end{align*}
where for $z\in \mathbb{C}^+$, $e_B=e_B(z)$ is the unique solution in $\mathbb{C}^+$ to the equation
\begin{align*}
e_B=\frac{1}{\int\frac{s}{1+\gamma e_Bs}dF_{w^2}(s)-z}.
\end{align*}
Also, by the same reasoning as in the proof of the Haar matrix case, the l.s.d. is supported on an interval bounded away from zero. This means that we can find the almost sure limits of the traces in terms of the Stieltjes transform of the l.s.d. at zero, or equivalently in terms of the inverse eta-transform:
\begin{align*}
\frac{1}{p}\tr(\frac{1}{n}R^{-1})&=\frac{n}{p}\tr[(Z_w^\top Z_w)^{-1}]\rightarrow\frac{\eta_{w^2}^{-1}(1-\gamma)}{\gamma},
\end{align*}
and therefore
$\tr(R^{-1})\xrightarrow{a.s.}\eta_{w^2}^{-1}(1-\gamma).$
Thus, from the expression of $h_{ii}$ given at the beginning, we also have
\begin{align*}
|h_{ii}-1+\frac{1}{1+w_i^2\eta_{w^2}^{-1}(1-\gamma)}|\xrightarrow{a.s.}0.
\end{align*}
Thus as $n$ goes to infinity, leverage-based sampling is equivalent to sampling $x_i$ with probability
\begin{align}
\pi_i=\min\left(\frac{r}{p}(1-\frac{1}{1+w_i^2\eta_{w^2}^{-1}(1-\gamma)}),1\right),
\end{align}
in the sense that $|\min(\frac{r}{p}h_{ii},1)-\pi_i|\xrightarrow{a.s.}0$. Therefore, it is not hard to see that the performance metrics we study have the same limits for leverage sampling and for sampling with respect to $\pi_i$. We argue for this in more detail below. 
Let $S^*$ be the sampling matrix based on the leverage scores, with diagonal entries $s_i^*\sim Bernoulli(\min(r/nh_{ii},1))$. This is the original sampling mechanism to which the theorem refers. Now, we have shown that $\|S-S^*\|_{op}\to0$ almost surely. Because of this, one can check that $\tr[Q_1(S,X)]-\tr[Q_1(S^*,X)]\to0$ almost surely. This follows by a simple matrix calculation expressing $A^{-1}-B^{-1}=-A^{-1}(B-A)B^{-1}$, and bounding the trace using Lemma \ref{lemma:bound trace(AB)}.

\subsection{Greedy leverage sampling}
\label{sec: greedy leverage supp}
As a direct corollary of Theorem \ref{theorem:sampling elliptical model}, we have the results for greedy leverage sampling. 
\begin{corollary}[Greedy leverage sampling] Under the conditions of Theorem \ref{theorem:sampling elliptical model}, suppose that for $p<r<n$, we take the $r$ rows of $X$ with the highest leverage scores and do linear regression on the resulting subsample of $X,Y$. Let $\tilde{w}^2=w^21_{[w^2>F_{w^2}^{-1}(1-\xi)]}$ denote the distribution of $F_{w^2}$ truncated at $1-\xi$.
Then
\begin{align*}
&\limn VE(\hbeta_s,\hbeta)=\frac{\eta_{\tilde{w}^2}^{-1}(1-\gamma)}{\eta_{w^2}^{-1}(1-\gamma)},
\\
&\limn VE(\hbeta_s,\hbeta)=1+\frac{1}{\gamma}\EE{w^21_{[w^2<F_{w^2}^{-1}(1-\xi)]}}\eta_{\tilde{w}^2}^{-1}(1-\gamma/\xi),\\
&\limn OE(\hbeta_s,\hbeta)=\frac{1+\EE{w^2}\eta_{\tilde{w}^2}^{-1}(1-\gamma)}{1+\EE{w^2}\eta_{w^2}^{-1}(1-\gamma)},
\end{align*}
where $\eta_{w^2}$ and $\eta_{\tilde{w}^2}$ are the $\eta$-transforms of $F_{w^2}$ and $F_{\tilde{w}^2}$, respectively, and the expectations are taken with respect to those limiting distributions.
\label{theorem:largest leverage scores}
\end{corollary}

\subsection{Table of tradeoff between computation and statistical accuracy}
\label{sec: table tradeoof supp}
We give a summary of the algorithmic complexity and statistical accuracy (variance efficiency) of each method in Table \ref{comp_stat}.

\begin{table}[]
\renewcommand{\arraystretch}{1}
\centering
\caption{Tradeoff between computation and statistical accuracy.}
\label{comp_stat}
\begin{tabular}{
>{\centering\arraybackslash}m{2.75cm}
>{\centering\arraybackslash}m{2.4cm}
>{\centering\arraybackslash}m{2.2cm}
>{\centering\arraybackslash}m{2.7cm}
>{\centering\arraybackslash}m{2.6cm}}
\toprule
Data matrix $X$ 
& Sketching matrix \centering$S$ 
& VE 
& Computational complexity
& Parallelization of sketching across $n$
\\\hline

 Incoherent, near-iid
& Uniform sampling 
& {\multirow{2}{*}{{\(\displaystyle\frac{n-p}{r-p}\)}}}
& $O(rp^2)$
& Embarassingly parallel
\\\cmidrule{1-2}\cmidrule{4-5}

Arbitrary
& Hadamard 
& 
&  $O(np\log n+rp^2)$ 
& Nontrivial
\\\midrule

Arbitrary
&iid entries 
& {$1+\frac{n-p}{r-p}$}  
& $O(rnp+rp^2)$
& Embarassingly parallel
\\
\bottomrule
\end{tabular}
\end{table}

\subsection{Simulation for leverage-based sampling}
\label{section:leverage:ex}
We consider a simple example where $w$ follows a discrete distribution, with 
$\PP{w_i=\pm d_1}$ = $\PP{w_i=\pm d_2}$ = $1/4$. $Z$ is a standard Gaussian random matrix and $\Sigma$ is the identity matrix. We plot simulation results as well as our theory for leverage score sampling, greedy leverage scores, uniform sampling and Hadamard projection. In the right panel, we also plot the histogram of the leverage scores of $X$. Our theory agrees very well with the simulations. 

We also observe that the greedy leverage sampling outperforms random leverage sampling, especially for relatively small $r$. Moreover, leverage sampling and greedy leverage scores have much better performances than uniform sampling. This is because the leverage scores are highly nonuniform in this example. 

In Figure \ref{Fig: compare leverage hadamard}, we also compare the theoretical performance of leverage score sampling and Hadamard projection in the same elliptical model, with several aspect ratios $\gamma$ and $d_1,d_2$. We skip the comparison with Gaussian/iid projection because the performance of Hadamard projection is uniformly better, as has been shown before. The difference between $d_1$ and $d_2$ is a measure of the non-uniformity of the data. 

When the data is relatively uniform (left panel), leverage sampling and Hadamard projection have similar VE. When in addition $r$ is small, leverage score sampling tends to perform better than Hadamard projection. However, when the dataset is nonuniform (right panel), leverage sampling and Hadamard projection can have very different performance. When $\gamma$ is small, leverage sampling works much better; but when $\gamma$ is large, Hadamard is uniformly better. Thus, when the dataset is nonuniform and the targeted dimension is rather small, leverage score sampling is the recommended method, provided that one can estimate the leverage scores efficiently. In conclusion, this example shows that the relative performance of sketching methods on elliptical data is quite complex, and perhaps one should mostly expect rules of thumb, instead of definitive answers.


\begin{figure}[]
\begin{subfigure}{.7\textwidth}
\includegraphics[width=\textwidth, center]{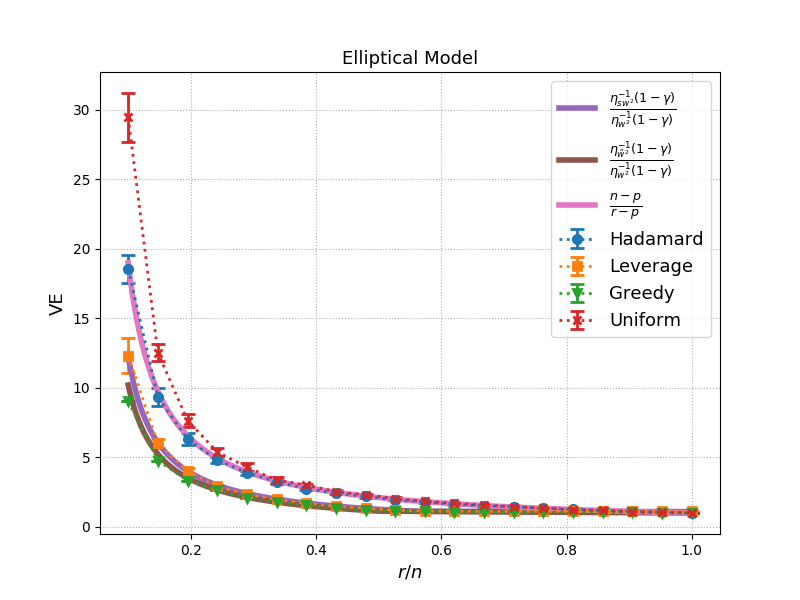}
\end{subfigure}
\begin{subfigure}{.28\textwidth}
\includegraphics[width=\textwidth, center]{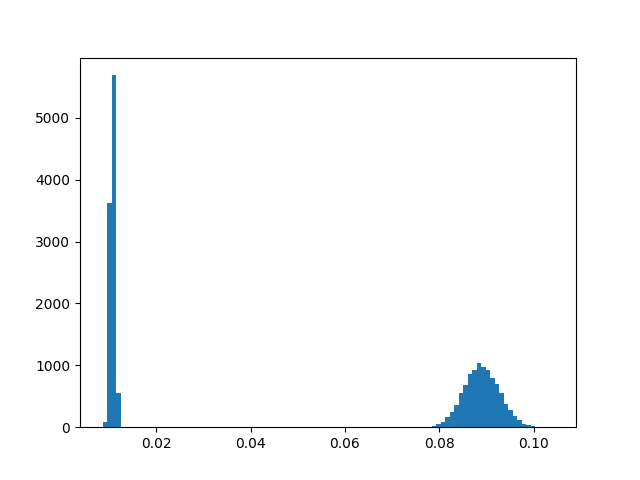}
\end{subfigure}
\caption{Leverage sampling, greedy leverage sampling and uniform sampling for elliptical model. We generate the data matrix $X$ from the elliptical model defined in \eqref{def:elliptical model}, and we take $d_1=1,d_2=3, n=20000, p=1000$ while $Z$ is generated with iid $\N(0,1)$ entries and $\Sigma$ is the identity. We let $r$ range from 4000 to 20000. At each dimension $r$ we repeat the experiments 50 times and take the average. For leverage sampling, we sample each row of $X$ independently with probability $\min(r/p\cdot h_{ii},1)$. For greedy leverage scores, we take the $r$ rows of $X$ with the largest leverage scores. For uniform sampling, we uniformly sample $r$ rows of $X$. We see a good match between theory and simulations.}
\label{fig:leverage_elliptical}
\end{figure}

\begin{figure}[]
\includegraphics[width=.7\textwidth, center]{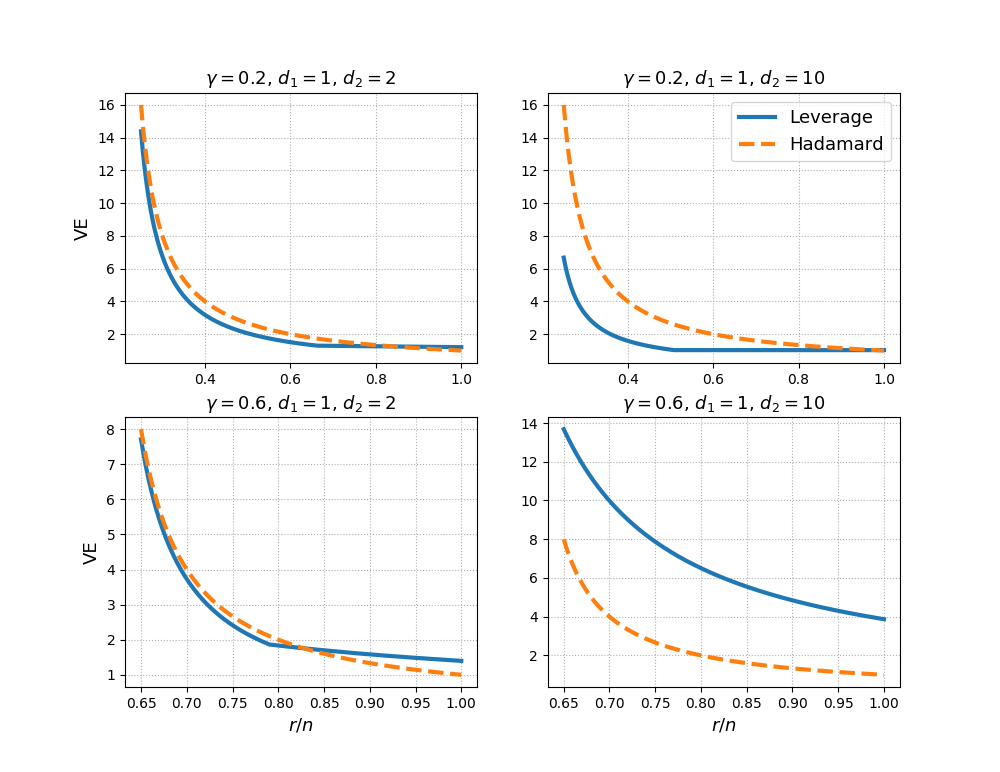}
\caption{Comparing leverage sampling and Hadamard projection.}
\label{Fig: compare leverage hadamard}
\end{figure}

Following is the details of the calculation.
\begin{align*}
&\eta_{w^2}(z)=\frac{1}{2}\frac{1}{1+zd_1^2}+\frac{1}{2}\frac{1}{1+zd_2^2},\\
&\eta_{sw^2}(z)=(1-\frac{1}{2}\min(\pi_1, 1)-\frac{1}{2}\min(\pi_2, 1))+\frac{1}{2}\min(\pi_1, 1)\frac{1}{1+d_1^2z}+\frac{1}{2}\min(\pi_2, 1)\frac{1}{1+d_2^2z},
\end{align*}
where 
\begin{align*}
\pi_1&=\frac{\xi}{\gamma}(1-\frac{1}{1+d_1^2\eta_{w^2}^{-1}(1-\gamma)}),\,\,
\pi_2=\frac{\xi}{\gamma}(1-\frac{1}{1+d_2^2\eta_{w^2}^{-1}(1-\gamma)}).
\end{align*}
It is easy to see that
\begin{align*}
\pi_1 + \pi_2 = 2\xi,
\end{align*}
and
\begin{align*}
\eta_{F_{w^2}}^{-1}(1-\gamma)&=\frac{1}{2d_1^2d_2^2}(-d_1^2-d_2^2+\frac{d_1^2+d_2^2}{2(1-\gamma)}+\sqrt{(d_1^2+d_2^2-\frac{d_1^2+d_2^2}{2(1-\gamma)})+\frac{4d_1^2d_2^2\gamma}{1-\gamma}}),
\end{align*}

If we use the $r$ rows of $X$ with the largest leverage scores, the truncated distribution $\tilde{w}$ in Theorem \ref{theorem:largest leverage scores} can be written as
\begin{align*}
F_{\tilde{w}^2}(t)=\begin{cases}\delta_{d_2^2}, & 0<\frac{r}{n}\leq\frac{1}{2}\\
(1-\frac{n}{2r})\delta_{d_1^2}+\frac{n}{2r}\delta_{d_2^2}, & \frac{1}{2}<\frac{r}{n}\leq1.
\end{cases}
\end{align*}
Therefore
\begin{align*}
\eta_{\tilde{w}^2}(z)=\begin{cases}\frac{1}{1+d_2^2z}, & 0<\frac{r}{n}\leq\frac{1}{2}\\
(1-\frac{n}{2r})\frac{1}{1+d_1^2z}+\frac{n}{2r}\frac{1}{1+d_2^2z}, &\frac{1}{2}<\frac{r}{n}\leq1,
\end{cases}
\end{align*}
thus
\begin{align*}
\eta_{F_{\tilde{w}^2}}^{-1}(1-\frac{\gamma}{\xi})=\begin{cases}
\frac{\gamma}{d_2^2(\xi-\gamma)}, & 0<\frac{r}{n}\leq\frac{1}{2}\\
\frac{1}{2d_1^2d_2^2}[-b+\sqrt{b^2+\frac{4d_1^2d_2^2\gamma}{\xi-\gamma}}], & \frac{1}{2}<\frac{r}{n}\leq1.
\end{cases}
\end{align*}
Here
\begin{align*}
b=d_1^2+d_2^2-\frac{(2\xi-1)d_2^2+d_1^2}{2(\xi-\gamma)}.
\end{align*}





\subsection{Simulation for nonuniform data}
\label{sec:nonuniform}
In Figure \ref{fig: t distribution}, each row of $X$ is generated from a $t$ distribution with 1 degree of freedom. Specifically, let $\Sigma$ be $p\times p$ covariance matrix with $\Sigma_{ij}=2\times2^{-|i-j|}$. Then each row of $X$ is generated as $\N(0,\Sigma)$ divided by a chi-squared random variable with 1 degree of freedom. We show the mean, as well as the 5\% and 95\% quantiles of VE over 1000 repetitions. We do not use standard deviation to illustrate the variability, because the variance can be rather large. 

We also plot the histogram of the leverage scores on the right. There are several extremely large leverage scores, which means that the design matrix is ill-conditioned. For readability's sake, we do not plot the results for uniform sampling and leverage sampling. Instead, we show them in Tables \ref{tab:uniform log VE} and \ref{tab:leverage log VE}. We observe the following:

\begin{itemize}
	\item Usually, the numerical mean of VE falls on the respective theoretical line. Moreover, the 95\% confidence intervals always cover the theoretical lines. This means that our results are correct on average.
	\item However, the VE can be anomalously large in some rare cases, driving the mean to be rather large. But even among the 1000 repetitions, the anomalous values only fewer than ten times. This explains why the standard deviations are large but the 90\% confidence intervals are relatively short.
	\item The reason for the abnormal phenomena is due to some rows of $X$ with large norms, which dominate the influence of $X$ on $Y$. When sketching the matrix, we shrink the influence of these dominating rows, either by mixing with other unimportant rows or by dropping them altogether. Therefore the sketched estimators lose too much accuracy.
	\item Even in this less favorable situation, the Hadamard transform is still the most desirable sketching method. It has small average VE, relatively small variability (i.e., short confidence intervals), and short running time.
\end{itemize}

\begin{figure}[]
\centering
\begin{subfigure}{.7\textwidth}
\includegraphics[width=5in]{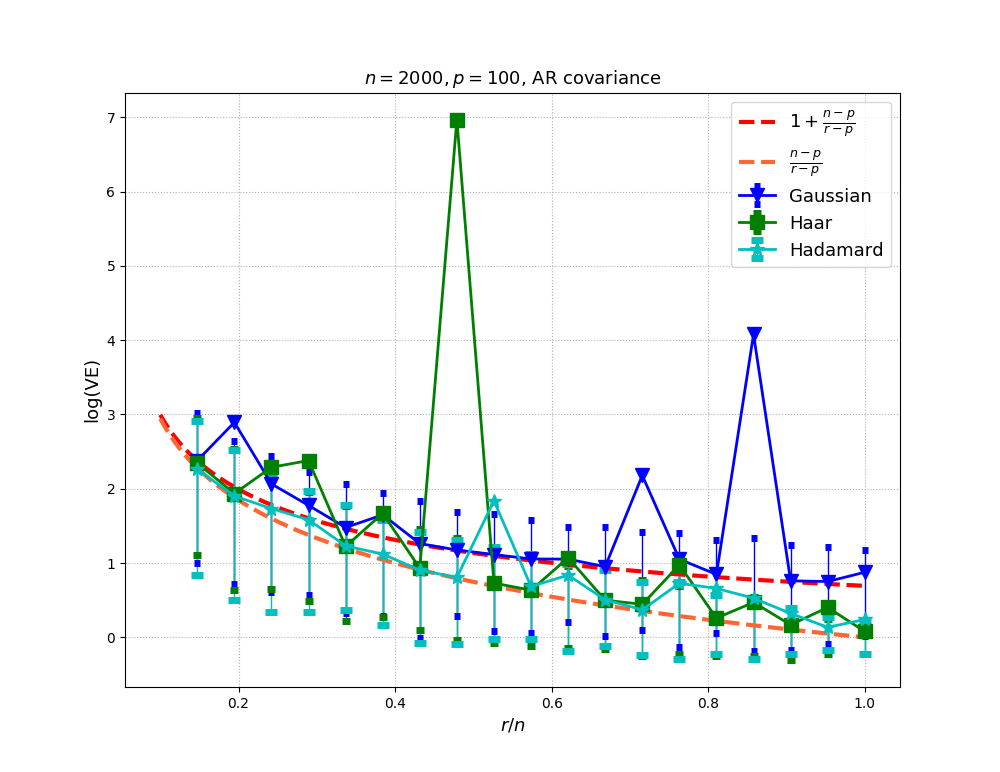}
\end{subfigure}\hspace{0.6cm}
\begin{subfigure}{.2\textwidth}
\includegraphics[width=2cm]{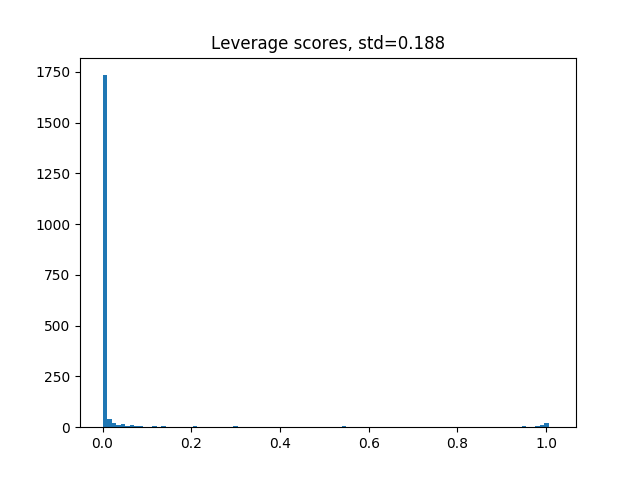}
\end{subfigure}
\caption{$t$ distribution.}\label{fig: t distribution}
\end{figure}

\begin{table}[]
\caption{Uniform sampling, $\log VE$}
\label{tab:uniform log VE}
\centering
\begin{tabular}{ccccc}
\toprule
   & mean        & 5\%          & 95\%        & 50\%        \\
\midrule
0  & 20.25823968 & 3.09919506   & 10.90630978 & 9.689192869 \\
1  & 17.26805584 & 2.044572843  & 9.33262643  & 7.881064556 \\
2  & 10.78396456 & 1.768101235  & 7.978181897 & 6.650869548 \\
3  & 11.63550561 & 3.043487316  & 7.369292418 & 5.715237875 \\
4  & 9.203440888 & 0.794976879  & 6.595984621 & 5.029794939 \\
5  & 8.980845283 & 1.326756793  & 5.774310548 & 4.38426635  \\
6  & 7.380001677 & 1.381715379  & 5.809281226 & 3.871467657 \\
7  & 5.175269082 & -0.182261694 & 5.605915977 & 3.399170722 \\
8  & 6.359148538 & 0.113763057  & 4.175648541 & 2.948032272 \\
9  & 9.15176239  & 1.078751974  & 4.24045247  & 2.574611614 \\
10 & 3.947147126 & 0.814135635  & 3.999019444 & 2.265773149 \\
11 & 3.44225402  & 0.122953471  & 3.484524911 & 1.932062314 \\
12 & 2.527325826 & 0.800658751  & 2.987471342 & 1.6527189   \\
13 & 1.824634546 & 0.322903628  & 2.587082545 & 1.331773852 \\
14 & 1.491822848 & 0.325396696  & 2.165972084 & 1.07823758  \\
15 & 1.04419024  & 0.130949401  & 1.589996964 & 0.81014184  \\
16 & 0.934085598 & 0.088131217  & 1.387329021 & 0.596745268 \\
17 & 0.873952782 & -0.052823305 & 1.034961439 & 0.375621926 \\
18 & 0.447688303 & -0.130411888 & 0.61016263  & 0.17981115  \\
19 & 6.978320808 & -0.189431268 & 0.226659269 & -4.75E-08 \\
\bottomrule 
\end{tabular}
\end{table}

\begin{table}[]
\caption{Leverage sampling, $\log VE$}
\label{tab:leverage log VE}
\centering
\begin{tabular}{ccccc}
\toprule
   & mean        & 5\%          & 95\%        & 50\%         \\
\midrule
0  & 7.626778204 & -0.159794491 & 0.821111174 & 0.204571271  \\
1  & 3.216688922 & -0.223342993 & 0.456512164 & 0.059604721  \\
2  & 1.015004939 & -0.207986859 & 0.424644872 & 0.029099432  \\
3  & 6.287155109 & -0.161735746 & 0.333680227 & 0.012818312  \\
4  & 0.882013996 & -0.238995897 & 0.397351467 & 0.010611862  \\
5  & 0.133943751 & -0.195882891 & 0.322395411 & 0.002172268  \\
6  & 0.487048617 & -0.160963212 & 0.256059749 & 0.002658917  \\
7  & 0.204928838 & -0.205892837 & 0.302131442 & 0.001159857  \\
8  & 1.354935903 & -0.212644915 & 0.376751395 & 0.001978275  \\
9  & 0.691138831 & -0.174171007 & 0.290014414 & 0.001378926  \\
10 & 3.129679099 & -0.250449933 & 0.432175622 & 6.24E-05     \\
11 & 0.40467726  & -0.280542607 & 0.403070117 & 0.000882543  \\
12 & 3.800307066 & -0.206858102 & 0.452128865 & -0.000171639 \\
13 & 0.403587432 & -0.18407553  & 0.251072808 & -7.39E-05    \\
14 & 0.758228813 & -0.296238449 & 0.452796686 & 0.000292609  \\
15 & 0.180781152 & -0.208918435 & 0.395642624 & 5.19E-05     \\
16 & 0.075991698 & -0.115281186 & 0.202626369 & 8.07E-05     \\
17 & 0.159661829 & -0.191824839 & 0.243641727 & 6.21E-05     \\
18 & 1.00887942  & -0.218371065 & 0.369474928 & 3.30E-05     \\
19 & 1.994998633 & -0.145457064 & 0.347448221 & 7.52E-08    \\
\bottomrule
\end{tabular}
\end{table}

\subsection{OE for two empirical datasets}
\label{sec:oe empirical supp}
See Figure \ref{fig :oe MSD flight} for the out-of-sample error on the two empirical datasets: Million Song Dataset (MSD) and the Flight Dataset.
\begin{figure}
\begin{subfigure}{.49\textwidth}
\includegraphics[width=\textwidth, center]{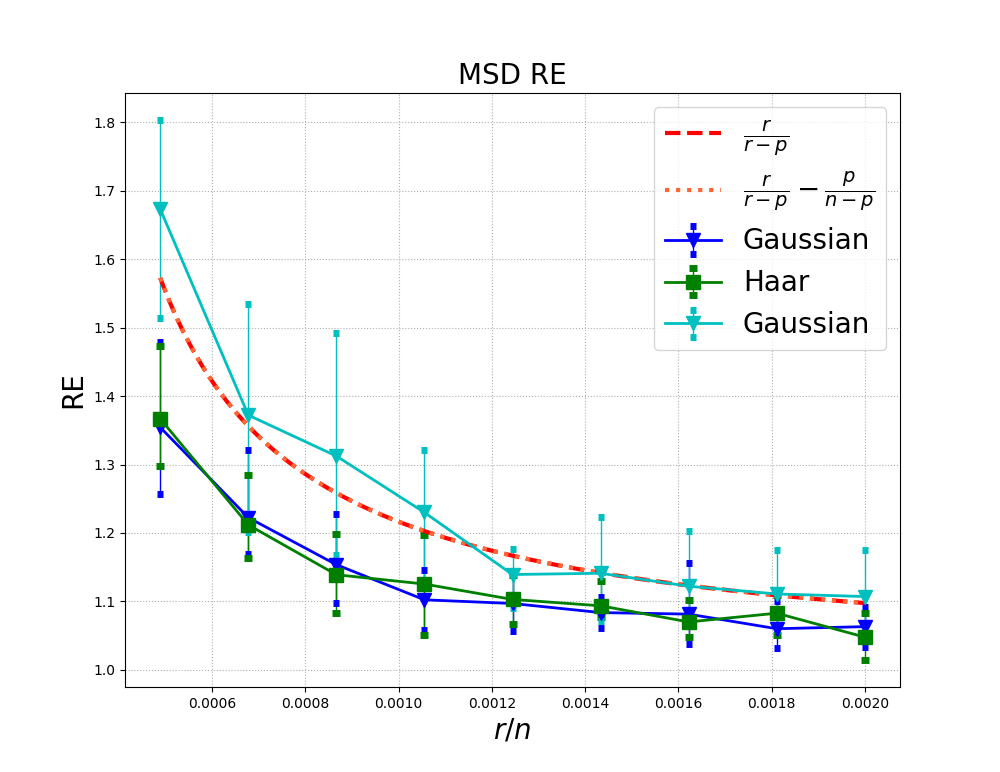}
\end{subfigure}
\hspace{0.05cm}
\begin{subfigure}{.49\textwidth}
\includegraphics[width=\textwidth, center]{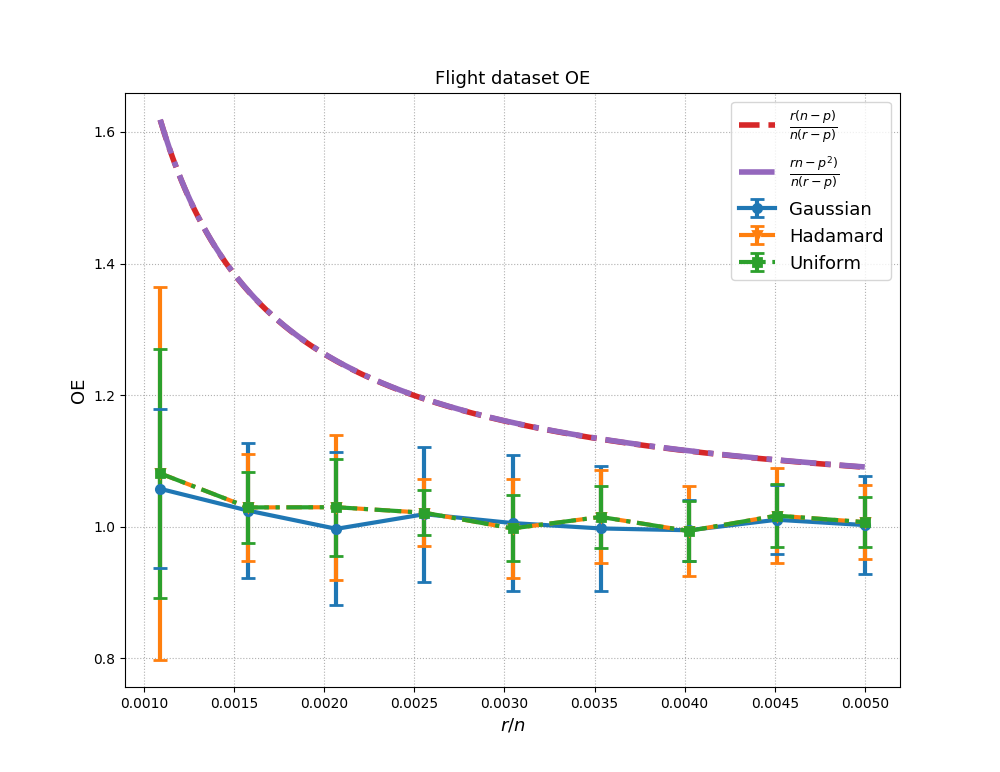}
\end{subfigure}
\caption{OE for MSD and flight dataset.}
\label{fig :oe MSD flight}
\end{figure}

\subsection{Comparison with previous bounds}
\label{section:compare mahoney}
We also compare our results with the upper bounds given in \cite{raskutti2014statistical}. For sub-Gaussian projections, they showed that if $r\geq c\log n$, then with probability greater than 0.7, it holds that
\begin{align*}
PE&\leq44(1+\frac{n}{r}),\,\,
RE\leq1+44\frac{p}{r}.
\end{align*}
For Hadamard projection, they showed that if $r\geq cp\log n(\log p+\log\log n)$, then with probability greater than 0.8, it holds that
\begin{align*}
PE&\leq1+40\log(np)(1+\frac{p}{r}),\,\,
RE\leq40\log(np)(1+\frac{n}{r}).
\end{align*}
In Figure \ref{mahoney}, we plot both our theoretical lines and the above upper bounds, as well as the simulation results. It is shown that our theory is much more accurate than these upper bounds.

\begin{figure}
\includegraphics[scale=0.6,center]{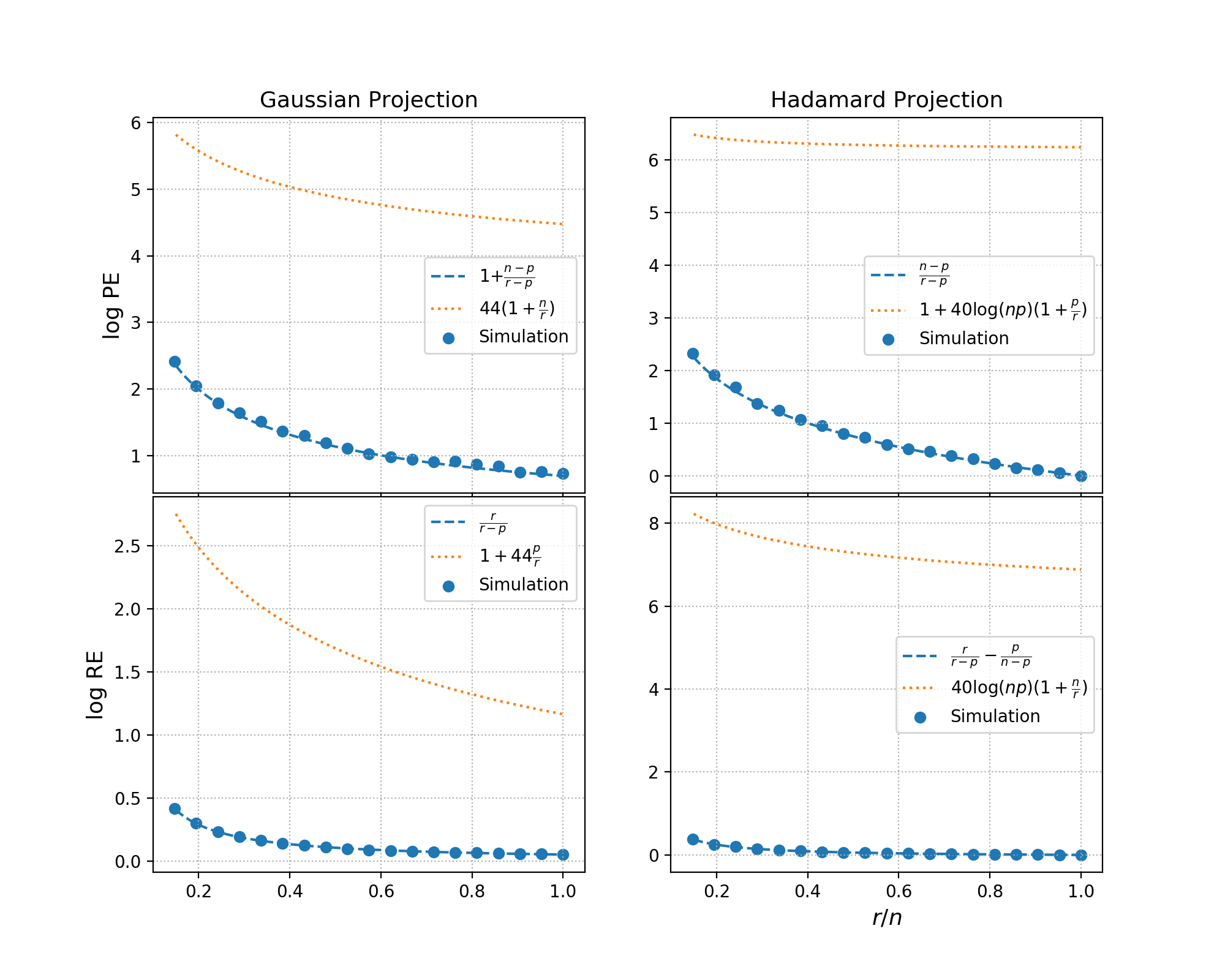}
\caption{Comparison with prior bounds. In this simulation, we let $n=2000$, the aspect ratio $\gamma=0.05$, with $r/n$ ranging from 0.15 to 1. The first column displays the results for $PE$ and $RE$ for Gaussian projection, while the second column shows results for randomized Hadamard projection. The $y$-axis is on the log scale. The data matrix $X$ is generated from Gaussian distribution and fixed at the beginning, while the coefficient $\beta$ is generated from uniform distribution and also fixed. At each dimension $r$ we repeat the simulation 50 times and all the relative efficiencies are averaged over 50 simulations. In each simulation, we generate the noise $\ep$ as well as the sketching matrix $S$. The orange dotted lines are drawn according to Section \ref{section:compare mahoney}, while the blue dashed lines are drawn according to our Theorem \ref{gsdx} and Theorem \ref{theorem:Hadamard S}.}
\label{mahoney}
\end{figure}

\subsection{Computation time}
\label{sec: computation time supp}
In this section we perform a more rigorous empirical comparison of the running time of sketching. We know that the running time of OLS has order of magnitude $O(np^2)$, while the running time of Hadamard projections is $c_1 np^2 + c_2 np\log(n)$ for some constants $c_i$. While the cubic term clearly dominates for large $n,p$, our goal is to understand the performance for finite samples $n,p$ on typical commodity hardware. For this reason, we perform careful timing experiments to determine the approximate values of the constants on a MacBook Pro (2.5 GHz CPU, Intel Core i7).

We obtain the following results. The time for full OLS and Hadamard sketching is approximately
\begin{align*}
t_{full}=4\times10^{-11}np^2,\,\,
t_{Hadamard}=2\times10^{-8}pn\log n+4\times10^{-11}rp^2
\end{align*}

See Figure \ref{cot} for a comparison of the running times for various combinations of $n,p$. For instance, we show the results for $n=7\cdot 10^4$, and $p=1.4\cdot 10^4$ with the sampling ratio ranging from $0.2$ to 1. We see that we save time if we take $r/n \le 0.6$.

We can also perform a more quantitative analysis. If we want to reduce the time by a factor of $0<c<1$, then we need
\begin{align*}
\frac{2\times10^{-8}pn\log n+4\times10^{-11}rp^2}{4\times10^{-11}np^2}\leq c
\end{align*}
or also $r\leq cn-500\frac{n\log n}{p}$, when $0<c-500\frac{\log n}{p}<1$. Then the out-of-sample prediction efficiency is lower bounded by
\begin{align*}
&OE(\hbeta_s,\hbeta)=\frac{r(n-p)}{n(r-p)}\\
\geq&\frac{n-p}{n}\left(1+\frac{p}{n(c-\frac{500\log n}{p})-p}\right)
=(1-\gamma)\left(1+\frac{\gamma}{c-\frac{500\log n}{p}-\gamma}\right).
\end{align*}
This shows how much we lose if we decrease the time by a factor of $c$.

Similarly, if we want to control the $VE$, say to ensure that $VE(\hbeta_s,\hbeta)\leq1+\delta$, then we need
\begin{align*}
r\geq\frac{n-p}{1+\delta}+p,
\end{align*}
then the we must spend at least a fraction of the full OLS time given below
\begin{align*}
\frac{r}{n}+\frac{500\log n}{p}\geq\frac{1-\gamma}{1+\delta}+\gamma+\frac{500\log n}{p}.
\end{align*}

\begin{figure}
\includegraphics[scale=0.7, center]{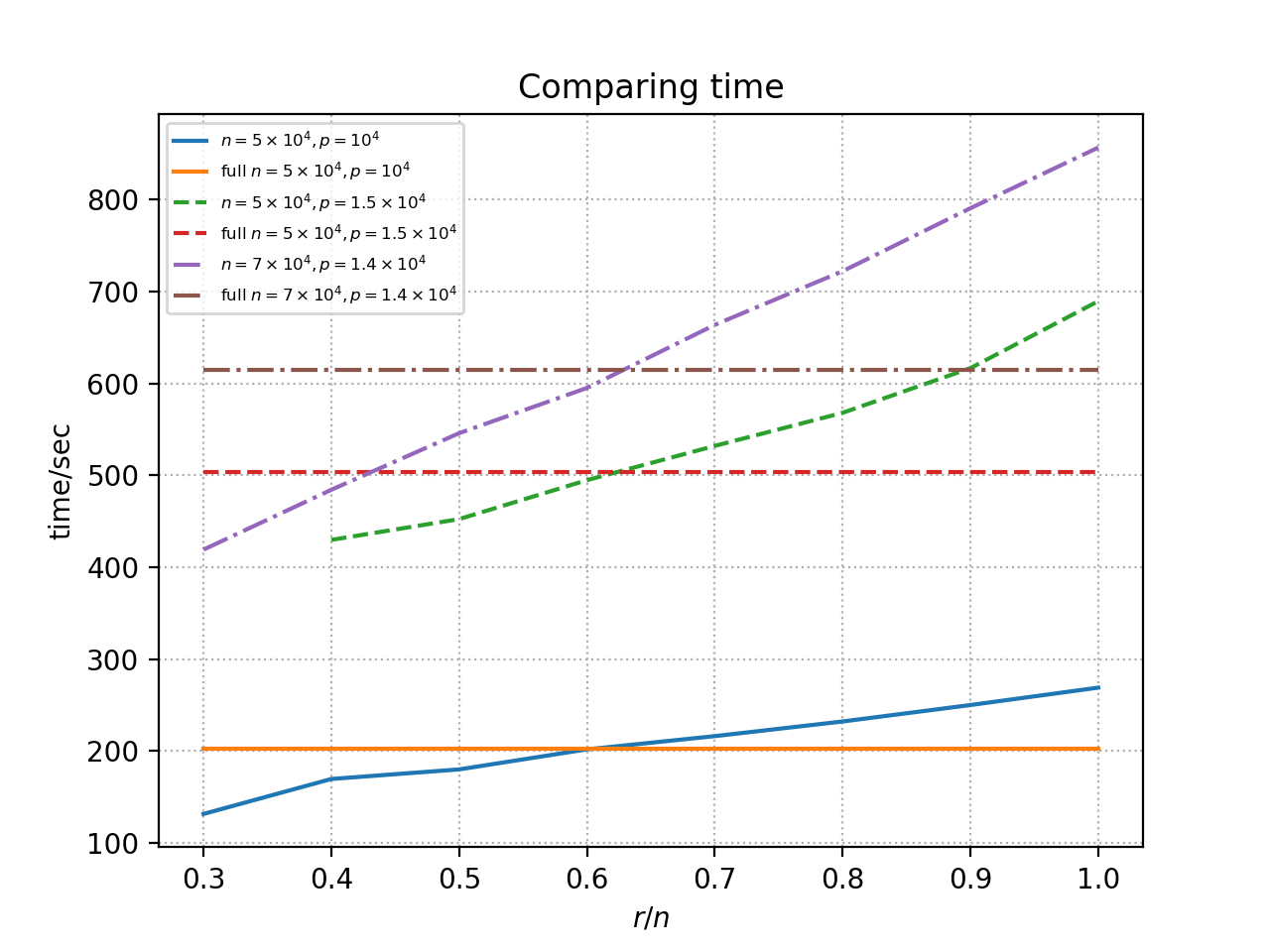}
\caption{A comparison of the running times for various combinations of $n,p$.}
\label{cot}
\end{figure}








\end{document}